\documentclass[12pt,reqno,a4paper]{amsart}
\usepackage{color}
\usepackage{amssymb}
\usepackage{eucal}
\usepackage{amsmath}
\usepackage{amsthm}
\usepackage{amscd}
\usepackage{graphics}
\usepackage{graphicx}
\usepackage{amsfonts}
\usepackage{latexsym}
\usepackage[all]{xy}
\usepackage{subfig}
\usepackage{float} 
\usepackage{graphicx} 
\usepackage{hyperref}

\newcommand\on{\operatorname}

\newcommand\oo{{\infty}}

\newcommand\cal{\mathcal}
\newcommand\R{\mathbb{R}}
\newcommand\T{\mathsf{T}}
\renewcommand\d{\mathrm{d}}
\newcommand \al{\alpha}

\newcommand\ta{\tau}

\newcommand\Om{\Omega}
\newcommand\M{\mathcal{M}}

\newcommand\pa{\partial}

\newcommand\SL{\on{SL}}
\newcommand\db{\on{DB}}
\renewcommand\sl{\mathfrak {sl}}

\newcommand\g{\mathfrak {g}}

\renewcommand\Im{\on{Im}}
\newcommand\Ker{\on{Ker}}

\newcommand\ind{\textnormal{ind}}
\newcommand\srp{\Pi^{\sharp}}

\newcommand\srm{\M^{\sharp}}

\numberwithin{equation}{section}
\newtheorem{thm}{\bf Theorem}[section]
\newtheorem{lem}[thm]{\bf Lemma}
\newtheorem{prop}[thm]{\bf Proposition}
\newtheorem{cor}[thm]{\bf Corollary}
\newtheorem{defn}{\bf Definition}[section]

\theoremstyle{remark}
\newtheorem{rem}{\bf Remark}[section]
\newtheorem{obs}{\bf Observation}[section]
\newtheorem{exmp}{\bf Example}[section]
\newtheorem{com}[thm]{\bf Comment}
\definecolor{bgd}{RGB}{153,0,51}      
\newtheorem{Conjecture}{\bf Conjecture}[section]

\keywords{Poisson manifold, Riemannian manifold, double bracket vector field, Lie Poisson structures, Poisson-Lie groups }

\subjclass[2020]{53D17, 58D17,	17B20}

\begin{document}
	
	\title[Metric Degeneracies and Gradient Flows on symplectic leaves]{Metric Degeneracies and Gradient Flows \\on symplectic leaves}
	
	\author[Z. Ravanpak]{Zohreh Ravanpak}\address{Z.\ Ravanpak: 	Department of Mathematics, West University of  Timi\c soara \\} 
	\email{zohreh.ravanpak@e-uvt.ro}
	\author[C. Vizman]{Cornelia Vizman}
	\address{C.\ Vizman:
		Department of Mathematics, West University of  Timi\c soara \\} \email{cornelia.vizman@e-uvt.ro}

	\begin{abstract}
		For a Poisson manifold endowed with a pseudo-Riemannian metric, we investigate degeneracies arising when the metric is restricted to symplectic leaves. Central to this work is the generalized double bracket (GDB) vector field—a geometric construct introduced in our earlier work—which generalizes gradient dynamics to indefinite metric settings. We identify admissible regions where the so-called double bracket metric remains non-degenerate on symplectic leaves, enabling the GDB vector field to function as a gradient flow on the admissible regions with respect to this metric. We illustrate these concepts with a variety of examples and carefully discuss the complications that arise when the pseudo-Riemannian metric fails to induce a non-degenerate metric on certain regions of the symplectic leaves.	
	\end{abstract}
	
	\maketitle

	\tableofcontents
	
	\section{Introduction}
	
	\medskip 
	Pseudo-Riemannian (or semi-Riemannian) manifolds generalize the Riemannian geometry by permitting indefinite metrics, characterized by a signature $(p,q)$ that specifies the numbers of positive and negative eigenvalues of the metric tensor. This flexibility enables applications ranging from spacetime modeling in general relativity—where Lorentzian metrics $(1,n-1)$ classify tangent vectors into timelike, spacelike, or null categories—to optimization on manifolds with non-positive definite geometries \cite{GaoLimYe2018}. However, indefinite signatures introduce challenges, such as metric degeneracies on submanifolds, which complicate geometric analysis and optimization. For comprehensive treatments of Riemannian geometry and optimization on smooth manifolds, see also \cite{Jost2017, Boumal2023}.

	The application of pseudo-Riemannian geometry to optimal transport problems establishes a profound connection between this advanced geometric framework and optimization challenges in information geometry and transport theory \cite{TK}. In a related vein, O'Neill \cite{O'Neill} explores the application of semi-Riemannian geometry to relativity theory, further demonstrating the versatility of these geometric structures in theoretical physics.
	
	A gradient vector field in a pseudo-Riemannian manifold is defined with respect to a smooth function $f$ and the metric tensor, producing the vector field $-\nabla f$ that points the direction of steepest descent. Such vector fields are fundamental in optimization on pseudo-Riemannian manifolds, enabling adaptations of algorithms like gradient descent for settings with indefinite metrics. In Lorentzian manifolds, for instance, the gradients of time functions\footnote{A function whose gradient is everywhere timelike.} help to define causal structures, distinguishing between timelike, spacelike, and null directions—an essential aspect of general relativity \cite{fathi,WK}.

	Unlike in Riemannian manifolds, where every smooth function has a uniquely defined gradient vector field with positive-definite metric properties, in pseudo-Riemannian manifolds the indefinite metric allows gradients to be timelike, spacelike, or null, and in some cases, the gradient may fail to be timelike everywhere or may not provide the same geometric control. In Poisson geometry, a related concept arises: for any smooth function $f$, the associated Hamiltonian vector field $X_f = \Pi^\sharp(df)$ (with $\Pi^\sharp$ induced by the Poisson bivector) provides a natural analogue of the gradient, aligning with the geometric structure of the manifold.
	
	In our previous work \cite{BRV}, we explored the framework of Poisson manifolds equipped with a Riemannian metric and introduced a vector field that we termed the generalized double bracket (GDB) vector field. This vector field is a gradient on symplectic leaves with respect to an induced metric known as the double bracket (DB) metric. Specifically, we generalized the setting from semi-simple compact Lie algebras $\mathfrak{g}$— the framework of double bracket vector fields \cite{brockett-1, brockett-2, bloch}— to Riemannian manifolds equipped with Poisson structures. Instead of restricting our attention to a Lie algebra $\mathfrak{g}$ equipped with negative-definite Killing form, we considered general Poisson manifolds that are equipped with a Riemannian metric.

	In this paper, we discuss the case where the metric is pseudo-Riemannian of indefinite signature. We explore GDB vector fields within the framework of pseudo-Riemannian manifolds, where complications arise due to the indefinite signature of the metric. An important question is whether the restriction of the metric \( g \) to the symplectic leaves of the Poisson structure is non-degenerate. For the special case of a noncompact semi-simple Lie algebra—where, in our approach, the Lie algebras do not need to be compact and the Killing metric is indefinite—this non-degeneracy occurs only on certain exceptional leaves. However, in general, the situation is more complex and this question will occupy the main part of this paper. We will particularly examine the case of the Lie algebra \( \mathfrak{sl}_2 \) and a large class of Poisson structures \( \Pi \) on \( \mathbb{R}^3 \), generalizing it non-linearly while maintaining the simple flat pseudo-metric \( g \) of \( \mathfrak{sl}_2 \) in the ambient space \( \mathbb{R}^3 \). This exploration will lead to an interesting and sophisticated interplay between \( \Pi \) and \( g \). 
	
	Our generalization necessitates a restriction to what we refer to as ``good symplectic leaves", specifically those for which the induced metric is non-degenerate. The concept of ``good leaves" is crucial when discussing GDB vector fields on symplectic leaves. These leaves satisfy certain regularity conditions, enabling the definition of a gradient vector field that behaves well with respect to the induced metric from the ambient pseudo-Riemannian structure.  For a Riemannian metric $g$ on $M$ , all symplectic leaves are good leaves. For a Poisson manifold $M$ with a metric $g$ of indefinite signature we show that this generalization is applicable only in regions where the metric induced on the symplectic leaves is non-degenerate. Metric degeneracy may occur not across the entire leaf but in localized regions of the leaf. In such scenarios, we work with lightlike leaves. We refer to those regions as ``green zones”, while leaves for which the induced metric is non-degenerate everywhere are ``good leaves". We will provide a characterization of such regions without explicitly determining the induced metric, which can be quite complex. This discussion will be illustrated through a wide class of Poisson structures on \( \mathbb{R}^3 \).
	
	Furthermore, in our setting a compatibility condition—generalizing the ad-invariance found in the Lie algebra case—between the metric and the Poisson structure on the manifold 
	$M$ is not necessary. Our formalism remains valid even without such an assumption, allowing for full generality. It's worth noting that the unimodularity requirement imposes a significant constraint on which Poisson manifolds can admit compatible Riemannian metrics \cite{Boucetta}. Unimodularity is a specific property of Poisson manifolds, and not all Poisson manifolds possess this characteristic. This observation underscores the broader applicability of our approach, which does not rely on such restrictive conditions.
	
	\smallskip
	\noindent {\bf Structure of the paper:} In Section \ref{sec:2}, we review the construction of GDB vector fields for an arbitrary Poisson manifold $M$ equipped with a Riemannian metric. Specifically: These vector fields possess two key properties: They are tangent to the symplectic leaves of $(M,\Pi)$; They are analogous to the double bracket vector fields in the linear case of a semi-simple Lie algebra $\g$. In Section \ref{sec:3}, we focus on the case where the metric $g$ has an indefinite signature. We introduce the concept of ``green zones" within symplectic leaves and prove a key theorem: the restriction of the GDB vector field to these green zones is the gradient of a smooth function with respect to the DB metric. In Section \ref{sec:4}, we introduce a broad class of Poisson structures on $\R^3$. This class encompasses notable examples such as the Lie-Poisson group $\sl_2^*$ and a related Poisson-Lie group as special cases.
	Section \ref{sec:5} is dedicated to an in-depth analysis of the symplectic leaves associated with the Poisson structures introduced in the previous section. Our investigation reveals that the leaf structure in these cases is significantly more complex than in the $\sl_2$ example. Notably, we propose the following conjecture: through careful selection of $\Pi$, it is possible to construct symplectic leaves of arbitrary genus. This result highlights the rich topological diversity present in these Poisson structures. In Section \ref{sec:6}, we introduce a pseudo-metric $g$ on the Poisson manifolds discussed in the preceding section. This enables us to investigate the challenges that arise when $g$ is pulled back to the symplectic leaves, with particular attention to points where non-degeneracy fails (the red lines $\mathcal{R}$). In Section \ref{sec:7}, we combine the information about symplectic leaves gathered in Section \ref{sec:4} with the analysis of the red zone. The intersection of a leaf $S$ with $\mathcal{R}$ yields what we refer to as ``red lines," while the remainder of the leaf, $S \setminus (\mathcal{R} \cap S)$, constitutes the previously mentioned green zones. Leaves that do not intersect $\mathcal{R}$ at all are called good leaves. It is precisely on these good leaves, or more generally within the green zones, that our main theorem applies. In the Sec.\ \ref{sec:8}, we finally compute the structures induced on the leaves for the class of our examples and illustrate the main theorem by means of them.

	\section {Metric degeneracies and gradient flows: definite signature}\label{sec:2}
	Let $(M,\Pi,g)$ be a Poisson manifold equipped with a Riemannian structure. In this section, we will review the basic definitions and gradient-like behavior of the GDB vector field on symplectic leaves, as presented in \cite{BRV}.
	
	\smallskip
	In the context of Poisson manifolds, for a smooth function \( f \) on symplectic leaves, the associated Hamiltonian vector field or symplectic gradient $X_f=\Pi^\sharp(\d f)$  satisfies
	$
	\omega(X_f,\cdot) = -\d f$, where $\Pi^{\sharp}$ is the Poisson tensor viewed as a map from covectors to vectors.
	In the context of Riemannian manifolds, the Riemannian gradient $\nabla f$ of a smooth function $f$ satisfies $g(\nabla f, \cdot) =\d f$\,.  
	
	Given a manifold  equipped with both Poisson and Reimannian structures, we have shown that the vector field $\pa _G:={(\Pi^{\sharp} \circ g^{\flat}) (X_G)} = i_{g^{\flat}(X_G)}\Pi$, where $g^{\flat}$ maps vectors to covectors using the metric tensor $g$, is a gradient vector field on symplectic leaves with respect to a so-called double bracket (DB) metric. We have termed this vector field the generalized double bracket (GDB) vector field. In the following, we briefly recall this construction. For more details, see: \cite{BRV}. 
	
	To better understand the geometry of GDB vector field, we introduced a symmetric contravariant 2-tensor field called the \emph{metriplectic tensor field} as: 
	\begin{equation}\label{metricplectic}
		\M(\alpha,\beta):={g}(\Pi^{\sharp}\al,\srp\beta)\,,\quad\alpha,\beta\in\Om^1(M)\,.
	\end{equation}
	Then the \emph{GDB vector field} can be defined using this metriplectic tensor field. For a smooth function $G$ on $M$, the GDB vector field denoted by $\partial_G$ is given by:\begin{equation}\label{defi}\partial_G:=-\M^{\sharp}(\d G)\,.\end{equation}
	This construction combines elements of Riemannian geometry, Poisson geometry, using metriplectic tensor field to create a new type of vector field. The GDB vector field encapsulates information from both the metric structure (through $g$) and the Poisson structure (through $\Pi$) of the manifold.
	
	Next, we need to define an appropriate metric on symplectic leaves. Let $S$ denote a symplectic leaf of $(M,\Pi,g)$. Since the metric on the ambient space is positive definite, the metriplectic tensor $\M$ is non-degenerate and so the 2-tensor field induced by $g$ on $S$, denoted as $g_{\ind}^S := \iota^* g$, is non-degenerate at each point $s\in S$. Here, $\iota \colon S \hookrightarrow M$ is the inclusion map. The desired \emph{double bracket (DB) metric} ${\tau}_{\db}^S$ interacts with the corresponding symplectic form $\omega$  on the symplectic leaf $S$ of $(M,\Pi,g)$ in the following way:
	\begin{equation}
		{\tau}_{\db}^{S}\left(X,Y\right):=({g}_{\ind}^{S})^{-1}\left({ i}_{X}\omega^S,{i}_{Y}\omega^S\right),\quad X,Y\in\mathfrak{X}(S)\,. \label{DB}
	\end{equation}
	\noindent With the above formulation, we have the following Theorem:
	\begin{thm}[\cite{BRV}]\label{gradient th}
		Let $(M,\Pi,g)$ be a smooth Poisson manifold equipped with a Reimannian structure. For any smooth function $G\in C^\oo(M)$,
		the GDB vector field $\pa_ G$, is a gradient vector field of $G|_{S}$ with respect to the DB metric:
		\begin{equation} \label{thmeq}
			(\pa_G)(x)=-\nabla_{\ta^S_{\db}} {(G|_{S})}(x),\quad x\in S\,. 
		\end{equation}
	\end{thm}

In \cite{BRV}, we incorporated a GDB vector field into Hamiltonian dynamics, this approach transforms stable equilibria into asymptotically stable equilibria, while preserving the structure of the symplectic leaves. We applied this method to the example of two harmonic oscillators in $(n:m)$ resonance.
	
	\section{Metric degeneracies and gradient flows: indefinite signature}\label{sec:3}
	In this section, we expand the GDB framework to incorporate metrics with indefinite signatures. This extension represents a significant advancement in the GDB theory.

	Let $(M,\Pi,g)$ be an $n$-dimensional smooth Poisson manifold endowed with a pseudo-Riemannian metric. Our discussion will focus on the geometry of GDB vector field in indefinite signatures. This case requires more careful consideration, as submanifolds might not be pseudo-Riemannian with respect to the induced metric \cite{O'Neill}. The degeneracy of $g^S_{\ind}$ occurs when $\Ker \srp \subsetneq \Ker \srm $, this is equivalent with the strict inclusion $\Im \srm \subsetneq \Im \srp $, see \cite{BRV}. Therefore, the obstacle for having a well-defined metric on a symplectic leave $S$ can be restricted to specific points or regions. The following discussion addresses this issue.
	
	\begin{defn} \label{Msing}
		\rm
		A point $m\in M$ is called \textbf{$\M$-regular} if $\Im \Pi^{\sharp} \vert_m = \Im \sharp_\M\vert_m$;  otherwise $m$ is referred to as \textbf{$\M$-singular}. A sympectic leaf $ S$ then is called a \textbf{good leaf} if all its points are $\M$-regular. 
	\end{defn}
	\noindent   Unlike a Riemannian metric $g$ that all points in $M$ are  $\M$-regular and $\M$-distribution $\Im\M^{\sharp}$ is integrable, in the case of a pseudo-Riemannian metric $g$, the integrability of this $\M$-distribution is not guaranteed. 
	Metric degeneracy on leaves occurs at specific forbidden points or regions, specifically at $\mathcal{M}$-singular points and within forbidden zones, the latter can be characterized as follows:
	\begin{defn}
		The set of points where the metriplectic tensor becomes degenerate is referred to as the \textbf{red zone}. We denote this set by $\mathcal{R}$, defined as:
	\begin{equation}
		\mathcal{R} = \left\{ m \in M : \exists v \in \Im(\Pi^{\sharp}_m) \setminus \{0\} \,\,s.t.\,\, g(v, w)=0 \,,\forall w\in \Im(\Pi^{\sharp}_m) \right\}\,.
	\end{equation}
	\end{defn}
	\begin{rem}
		We note that $\Pi$-singular point-like leaves or higher-dimensional $\Pi$-singular leaves may not necessarily contain forbidden zones.
	\end{rem}
	\noindent By introducing and analyzing the red zone, we provide a more comprehensive picture of how the interplay between the Poisson structure and the pseudo-metric affects the geometry and analysis on these manifolds. This approach allows for understanding of where our methods are applicable.
	Consequently, the admissible parts are either good leaves or $\mathcal{M}$-regular parts of symplectic leaves, the latter can be characterized as follows:
	\begin{defn}
		The region $ \mathcal S=S\backslash (\mathcal R \cap S)$ of a symplectic leaf $S$, obtained by subtracting the intersection of red zone with the symplectic leave, is called the \textbf{green zone} of $S$.
	\end{defn}
	\begin{rem}
		A good leave is then as a leaf for which $\mathcal R \cap S=\emptyset$. 
	\end{rem}
	\begin{rem}
		It can be observed that the two-tensor $g_{\ind}^S := \iota^* g$ induced by $g$ on $S\subset M$ is degenerate at $s\in S$ if and only if $s$ is $\M$-singular. Consequently, for a good leaf the induced metric $g^S_{\ind}$ is non-degenerate.
	\end{rem}
	To investigate the interrelation between the Poisson structure $\Pi$ and the metric $g$ on $M$, we restrict $g$ to the symplectic leaves determined by $\Pi$, examining the degeneracy of the induced metric on each leaf. The classification of these degeneracies yields information about geometric interaction between $\Pi$ and $g$.
	
	\smallskip
	Let us assume that the two-tensor $(g_{\ind}^S)_x$ is degenerate on $\T_xS$ of a symplectic leaf $S$.  There exists a non-zero vector $\xi \in \T_x S$, such that $(g_{\ind}^S)_x(\xi, v)=0, \quad \forall v \in \T_x S$. The \textbf{radical or null space} of $\T_x S$ \cite{DB}, with respect to $g_{\ind}^S$, is a subspace $\operatorname{Rad} \T _x S$ of $\T _x S$ defined by 
	\begin{equation}\label{nullD}
		\operatorname {Rad}\T_x S=\left\{\xi \in \T_x S ; (g_{\ind}^S)_x(\xi, v)=0, \forall v \in \T_x S\right\}\,.
	\end{equation}
	\begin{defn} (\cite{DuBe})
		We say a symplectice leaf $S$ of a Poisson Manifold $(M, \Pi)$ is a \textbf{lightlike leaf} if the mapping 
	\begin{equation}
		\begin{array}{rcl}
			\Delta: S & \to& \operatorname {Rad} \T S\\
			x &\mapsto& \Delta_x:=\operatorname {Rad}\T_x S\,,
			
		\end{array}
		\end{equation}
		defines a nonzero differentiable distribution on $S$ . It is called the lightlike distribution on $S$. The degree of nullity $r$ of $S$ is defined as the dimension of the fibers of $\Delta$, i.e., the dimension of $\operatorname{Rad} \T_x S$.
	\end{defn}   
	
		\begin{rem}
		A lightlike leaf is then as a leaf for which $\mathcal R \cap S\neq \emptyset$. Therefore, it is necessarily contained within both the green and red zones.
	\end{rem}	
	Therefore, a lightlike leaf requires special consideration. Since $\mathcal{M}$ is non-degenerate in the green zones, it follows that $g_{\mathrm{ind}}^S$ is non-degenerate on $\T\mathcal{S}$. In this case, we denote the induced metric by $g_{\mathrm{ind}}^{\mathcal{S}}$. 
	
	Based on the preceding discussion, we conclude that a symplectic leaf $(S,\omega)$ of a Poisson manifold $(M,\Pi,g)$ with a pseudo-Riemannian metric can be characterized as either a good leaf, a bad leaf, or a lightlike leaf,  depending on the nullity of $S$. The lemma below is a constrained result of these observations. 
	\begin{lem}
Let $(M,\Pi,g)$ be a Poisson manifold with a pseudo-Riemannian metric. The degeneracies of the induced metric on symplectic leaves can be encoded into three distinct classes as follows:
		\begin{enumerate} 
			\item If $r=0$, $S$ is a good leaf and DB metric is defined as in (\ref{DB}).
			\item If $r=\operatorname {dim}(S)$, $S$ is a bad leaf, a well-defined metric on $S$ does not exist, as it is degenerate across the entirety of $\T S$.
			\item If $r<\operatorname {dim}(S)$, $S$ is a lightlike leaf, the DB metric is well-defined on the green zones of $S$ as 
		\end{enumerate}
		\begin{equation}
			\tau_{\text{DB}}^{\mathcal{S}}(X,Y) := (g_{\text{ind}}^{\mathcal{S}})^{-1}(i_X \omega^{\mathcal{S}}, i_Y \omega^{\mathcal{S}}), \quad X,Y\in \T \mathcal S\,. \label{DBin}
		\end{equation}
	\end{lem}
	\noindent We note that the red zone lead to genuine singularity in the DB metric. Therefore, the space loses its metric properties locally, creating a `hole' in the metric structure.
	
	\begin{rem}
		The green zone of a lightlike leaf includes regions where the induced metric $g_{\ind}^{\mathcal S}$, has both definite and indefinite signatures, corresponding to the Riemannian (not necessarily Euclidean)
		green zone ${\mathcal S}$ and 
		pseudo-Riemannian with indefinite metric 
		(not necessarily Lorentzian)
		green zone ${\mathcal S}$, respectively.
	\end{rem}
	\begin{proof}
		When restricting the null cone of $\T_xM$ to a symplectic leaf $S$, we define $\mathcal N_x:=\{v\in \T_x S-\{0\}:\,\, g_{\ind}^S(v,v)=0\}$.  We always have $\mathcal N_x  \subseteq \operatorname{Rad}\T_x S $. The Riemannian green zone ${\mathcal S}$ is the region where $\operatorname{Rad}\T_x S=\mathcal N_x$ for all $x\in {\mathcal S}$. The pseudo-Riemannian green zone ${\mathcal S}$ is the region where $\mathcal N_x  \subset \operatorname{Rad}\T_x S$ for all $x\in {\mathcal S}$.
	\end{proof}

	Note that, in the case of three-dimensional Poisson manifolds, Riemannian green zones are regions ${\mathcal S}^{E} \subset S$ where the induced metric  has Euclidean signature, while pseudo-Riemannian green zones ${\mathcal S}^{L} \subset S$  are regions where the induced metric has Lorentzian signature.
	
	Finally, Theorem \ref{gradient th} for indefinite signatures reads as follows:
	\begin{thm}\label{gradient ps}
		Let $(M,\Pi,g)$ be a smooth Poisson manifold equipped with a psuedo-Reimannian structure then:
		\begin{enumerate}
			\item On a good leaf $S$, the GDB vectore feild is a  gradient vector field with respect to the DB metric.
			\item On a bad leaf, DB metric is undefined and consequently, the GDB vector field is also not defined there.
			\item On a lightlike leaf $S$, the GDB vectore field is a gradient vector field of $G|_{\mathcal S}$ on each green zone $\mathcal S$ with respect to the DB metric:
			\begin{equation} \label{thmeq1}
				(\pa_G)(x)=-\nabla_{\ta^{\mathcal S}_{\db}} {(G|_{\mathcal S})}(x),\quad x\in \mathcal S\,. 
			\end{equation}
			
		\end{enumerate}    
	\end{thm}
	\begin{proof}
		(1)	The proof is analogous to that for the Riemannian case presented in our previous work \cite{BRV}. (2) is obvious. (3) This result follows from the fact that $g_{\mathrm{ind}}^S$, and therefore the DB metric, is non-degenerate on $\T\mathcal{S}$. Moreover, the non-degenerate induced metrics $g_{\mathrm{ind}}^{\mathcal{S}}$ on the green zones of a lightlike leaf have the same signature as the DB metrics $\tau_{\mathrm{DB}}^{\mathcal{S}}$. This is evident because the co-metric tensor inherits properties from the metric, including its signature. For the pseudo-Riemannian green zone, for a smooth function $G$, suppose $X_G(x)\in \T_x\mathcal S$ is a null vector with respect to $g_{\ind}^{{\mathcal S}}$, then it is a null vector with respect to $\tau_{\text{DB}}^{\mathcal{S}}$:
		\begin{equation}
		[g_{\ind}^{{\mathcal S}}]_{(x)}(X_G,X_G)=0\quad \Longrightarrow\quad  [\tau_{\text{DB}}^{{\mathcal S}}]_{(x)}(X_G,X_G)=[(g_{\ind}^{{\mathcal S}})^{-1}]_{(x)}(\d G,\d G)=0\,.
		\end{equation}
		The tangent space to a symplectic leaf at any point is spanned by Hamiltonian vector fields tangent to the leaf, which are generated by functions constant on the leaf. Thus, the conclusion follows.
	\end{proof}
	
\noindent {\bf Trajectories of GDB vector fields.}	As we have observed, to any smooth function on a Poisson manifold $(M,\Pi,g)$ equipped with a metric, we can associate a Hamiltonian vector field and a GDB vector field. Let us discuss the trajectories of Hamiltonian vector fields and those of GDB vector fields.
	The conservation of the Hamiltonian implies that if we evaluate the Hamiltonian along the flow generated by  $X_G$, it remains constant over time. In contrast, this behaviour changes for the flow of GDB of a Hamiltonian function as follows:

		\begin{lem}[Geometric interpretation of the GDB vector field: Riemannian metric]\label{extrema}
			Let $X_G$ be the Hamiltonian vector field of a smooth function $G$ on a Poisson manifold $(M, \Pi, g)$, where $g$ is a Riemannian metric. Then, on each symplectic leaf, the flow of the GDB vector field $
			\partial_G := \Pi^{\sharp}(g^{\flat}(X_G))$
			strictly decreases the value of $G$ and drives the system toward the critical points (extrema) of $G$.
		\end{lem}
		
		\begin{proof}
			Let $\gamma(t)$ be an integral curve of the GDB vector field on a symplectic leaf $S$, i.e.
			\[
			\dot{\gamma}(t) = \partial_G|_{\gamma(t)}\,.
			\]
			The rate of change of $G$ along this flow is
			\begin{equation}\label{flow}
			\frac{d}{dt} G(\gamma(t)) = dG(\dot{\gamma}(t)) = \tau_{\mathrm{DB}}(-\nabla_{\tau_{\mathrm{DB}}}(G|_{S}), \dot{\gamma}(t))
			= \tau_{\mathrm{DB}}(-\partial_G, \partial_G)
			= -\|\partial_G\|^2_{\tau_{\mathrm{DB}}} \leq 0\,.
			\end{equation}
			The last inequality holds because $\tau_{\mathrm{DB}}$ is positive definite, i.e., $S$ is a good leaf. In this case, as we follow the flow generated by $\partial_G$, the function $G$ strictly decreases, except at critical points where $\partial_G = 0$. Thus, the flow acts as a geometric analogue of gradient descent, ensuring that trajectories move toward points where $\partial_G = 0$. Note that since $\partial_G = (\Pi^{\sharp} \circ g^{\flat})(X_G)$, the critical points of $X_G$ are also critical points of $\partial_G$.
		\end{proof}
		
	Therefore, the GDB vector field $\partial_G$ defines a geometric flow that unites the symplectic (Hamiltonian) structure and the Riemannian metric to realize a ``steepest descent'' of $G$ on each symplectic leaf. Analogous to classical gradient descent, the GDB flow strictly decreases $G$ (except at critical points) with respect to the geometry determined by $\Pi$ and $g$, ensuring that trajectories asymptotically converge to the critical points of $G$. This mirrors the Lyapunov function criterion for stability: $G$ serves as a strict Lyapunov function for the GDB flow, (see \cite{BRV}).
		
		\medskip
		Lightlike manifolds with one-dimensional lightlike distributions are important objects that have been extensively studied in both mathematics and physics. Geometrically, their lightlike distributions are naturally integrable, which is a valuable property. Such manifolds frequently serve as models for singular regions in the spacetime of general relativity.
		
		In what follows, we focus on three-dimensional Poisson manifolds. In this context, the maximal symplectic leaves are two-dimensional. Consequently, lightlike leaves in our setting either possess one-dimensional lightlike distributions or are zero-dimensional (i.e., point-like leaves). 
		
			\begin{rem}[Geometric interpretation of the GDB vector field: pseudo-Riemannian metric]\label{extremap}
				Let $(M, \Pi, g)$ be a three-dimensional Poisson manifold equipped with a pseudo-Riemannian metric $g$, and let $G$ be a smooth function on $M$. A lightlike symplectic leaf  $S$ contains both Euclidean green zones ${\mathcal S}^E$ and Lorentzian green zones ${\mathcal S}^L$. On the Euclidean green zones ${\mathcal S}^E$, the GDB vector field behaves as in Lemma~\ref{extrema}: the flow strictly decreases $G$ and trajectories converge to the critical points of $G$.
				
				On the Lorentzian green zones ${\mathcal S}^L$, the sign of the rate of change of $G$ along the GDB flow at each point $s \in {\mathcal S}^L$ is determined by the causal character (spacelike, timelike, or null) of $\partial_G|_s$, see (\ref{flow}):
				\begin{itemize}
					\item If $\partial_G|_s$ is spacelike, $G$ decreases along the flow;
					\item If $\partial_G|_s$ is timelike, $G$ increases along the flow;
					\item If $\partial_G|_s$ is null (lightlike), $G$ is constant along the flow.
				\end{itemize}
				Therefore, unlike the Riemannian case, the GDB flow in Lorentzian regions does not provide a globally defined descent direction for $G$; instead, the behavior is determined pointwise by the signature of the DB metric and the local causal character of the GDB vector field.
			\end{rem}

		\noindent The geometric interpretation of the Theorem \ref{gradient ps} and Lemma \ref{extremap} in the context of the Lorentzian green zone requires special investigation, which we will undertake below.
		
	\begin{defn}
		Let ${\mathcal S}^L$ be a Lorentzian green zone of a Poisson manifold equipped with a pseudo-Riemannian metric. A GDB vector field $\partial_G$ is called \textbf{globally null} on ${\mathcal S}^L$ if
	$
		\tau_{\mathrm{DB}}^{{\mathcal S}^L}(\partial_G, \partial_G) = 0
		$
		at all points of ${\mathcal S}^L$. 	We say that $\partial_G$ is \textbf{locally null} on ${\mathcal S}^L$ if its flow can encounter points where $\partial_G$ lies on the null cone of the metric, that is, there exist points in ${\mathcal S}^L$ where
		$
		\tau_{\mathrm{DB}}^{{\mathcal S}^L}(\partial_G, \partial_G) = 0.
		$
	\end{defn}
		
\noindent

Consequently, the descent property of the GDB flow fails for a globally null GDB vector field, as $G$ remains constant along its trajectories. For a locally null GDB vector field, the descent property fails precisely at points where the flow encounters the null cone, complicating optimization by not providing a valid descent direction at those points.

\begin{rem}
	Let $(M, \Pi, g)$ be a smooth Poisson manifold equipped with a pseudo-Riemannian metric. If the GDB vector field $\partial_G$ of a smooth function $G$ is proportional to the Hamiltonian vector field $X_G$, then $\partial_G$ is a globally null GDB vector field.
\end{rem}
\begin{proof}
	Suppose $\partial_G$ is proportional to $X_G$, i.e., $g^{\flat}(X_G) = k\, dG$ for some constant $k$. By the definition of the Hamiltonian vector field, $X_G(G) = 0$ and the proportionality implies  $\pa_G(G)=0$ at all points. This means that $\partial_G$ is tangent to the level sets of $G$ and so  $\tau_{\mathrm{DB}}^{{\mathcal S}^L}(\partial_G, \partial_G) = 0$ at all points. Therefore, $\partial_G$ is a globally null GDB vector field.
\end{proof}

\begin{lem}
	On green zones, the Hamiltonian vector field $X_G$ of a function $G$ is orthogonal to the GDB vector field $\partial_G$ of $G$ with respect to the DB metric.
\end{lem}

\begin{proof}
	By definition, the musical isomorphism for the DB metric gives
	\begin{equation}
	(\tau_{\mathrm{DB}})^{\flat}(\partial_G) = -\left[\omega^{\flat} \circ (g_{\mathrm{ind}}^{\flat})^{-1} \circ \omega^{\flat}\right]\left(\Pi^{\sharp} \circ g^{\flat}\right)(X_G) = \omega(X_G, \cdot).
	\end{equation}
	Therefore, $\tau_{\mathrm{DB}}(\partial_G, X_G) = (\tau_{\mathrm{DB}})^{\flat}(\partial_G)(X_G) = \omega(X_G, X_G) = 0\,.$
\end{proof}

\noindent
Note that in the Euclidean green zone ${\mathcal S}^E$, this orthogonality corresponds to the usual notion of perpendicularity. In the Lorentzian green zone ${\mathcal S}^L$, however, null) vectors can be self-orthogonal.

		\section{Symplectic foliation of \texorpdfstring{$\mathbb{R}^3$}{R3}  induced by a class of Poisson structures} \label{sec:4}

		In this section, we consider a broad class of Poisson structures on $\R ^3$ that illustrate all the new concepts introduced in the previous sections. Among these, we present three specific examples: a linear Poisson structure, a quadratic Poisson structure, and a Poisson Lie group. The symplectic leaves and green zones for each case, where our main theorem is applicable, are carefully examined.
		
		\medskip
		The canonical basis of the Lie algebra $\sl_2$,
		$${\bf e}_1=\left(\begin{array}{cc}
			0&1\\
			0&0
		\end{array}\right),~~~{\bf e}_2=\left(\begin{array}{cc}
			0&0\\
			1&0
		\end{array}\right),~~~{\bf e}_3=\left(\begin{array}{cc}
			1&0\\
			0&-1
		\end{array}\right),$$
		gives rise to the standard form of its Lie brackets: 
		\begin{equation}
			[{\bf e}_1,{\bf e}_2]={\bf e}_3; \qquad [{\bf e}_1,{\bf e}_3]=-2{\bf e}_1; \qquad [{\bf e}_2,{\bf e}_3]=2{\bf e}_2. \label{sl2}
		\end{equation}
		They induce a linear Poisson structure on the dual. More precisely, rescaling the basis elements, $x := {\bf e}_1/ \sqrt{2}$, $y:= {\bf e}_2 / \sqrt{2}$, and $z:= {\bf e}_3 /2$, and considering them as coordinates on the dual $\sl(2,\mathbb{R})^* \cong \R^3$, we get 
		\begin{equation}
			\{ x , y \}_{\mathrm{lin}}=z \: , \qquad  \{z , x \}_{\mathrm{lin}}= x \: , \qquad \{z , y \}_{\mathrm{lin}}= -y \: .\label{sl2brackets}
		\end{equation}
		These can be considered as the fundamental Poisson brackets on $\R^3$, extended to all smooth functions by means of the Leibniz rule. 
		
		\medskip \noindent
		We now consider the following non-linear generalization of these brackets:
		\begin{equation}
			\{ x , y \}= U(z) + V(z) \, xy   \: , \qquad  \{z , x \}= x \: , \qquad \{z , y \}= -y \: ,\label{sl2nonlin}
		\end{equation}
		where $U$ and $V$ are arbitrarily chosen smooth functions. For the choice $U = \mathrm{id}$ and $V=0$ we regain the  formulas  \eqref{sl2brackets}. It is easy to verify that \eqref{sl2nonlin} satisfies the Jacobi identity and thus defines a Poisson structure. Such Poisson structures appeared in the study of two-dimensional gravity models \cite{TT}.

		\medskip \noindent
		A nice feature of these brackets is that a Casimir function, a non-constant function in the center of the Poisson bracket, can be found explicitly: 
		\begin{lem}[\cite{TT}] \label{lemC}
			Let $P$ be a primitive of the function $V$, $P'(z)=V(z)$, and $Q$ 
			such that $Q'(z) = U(z) \exp(P(z))$. Then the function $ \mathcal C \in C^\infty(\mathbb{R}^3)$ defined by
			\begin{equation} \label{C}
				\mathcal C(x,y,z) := xy \exp{\left(P(z)\right)} + Q(z)
			\end{equation}
			is a Casimir function of the brackets \eqref{sl2nonlin}.
		\end{lem}
		\noindent The generic symplectic leaves are obtained simply from putting $ \mathcal C$ equal to some constant $c \in \R$, $\mathcal C(x,y,z):= c$. This will permit us to  visualize the leaves in different cases. 
		
		\smallskip \noindent
		While the generic leaves are two-dimensional, there also exist point-like singular leaves. These singular leaves occur when the right-hand side of equation \eqref{sl2nonlin} vanishes, specifically when $(x,y)=(0,0)$ and $z$ is a root of the function $U$. Consequently, the singular symplectic leaves are confined to the $z$-axis and are located at these specific $z$ values. We denote the union of all singular symplectic leaves of the Poisson manifold $(M,\Pi)$ by
		\begin{equation} \label{Lsing}
			{\cal L}_{sing} := \{(0,0,z) \in \R^3 \vert U(z)=0\} \, .
		\end{equation}
		All other leaves are regular. 
		
		Let us now depict some of the symplectic leaves in particular cases: 
		\begin{exmp}[Linear brackets] \label{sl2leaves}
			For the brackets \eqref{sl2brackets}, the Casimir function \eqref{C} becomes 
			\begin{equation}
				\mathcal C_{\mathrm{lin}}(x,y,z)=xy + \tfrac{1}{2}z^2 \label{Clin}.
			\end{equation} 
			To plot some of its well-known level surfaces, it is convenient to use the  coordinates 
			\begin{equation}\label{XYT}
				X := z \: ,
				\quad Y := \frac{x + y}{\sqrt{2}} \: , \quad T := \frac{x - y}{\sqrt{2}} \,,
			\end{equation}
			which gives
			\begin{equation}
				2\mathcal C_{\mathrm{lin}}= X^2 + Y^2 -T^2 \, .\label{Clin2}
			\end{equation}
			In this example, there exists precisely one singular symplectic leaf, located at the origin. The level surface $\mathcal C_{\mathrm{lin}}=c$ for $c=0$ divides into three distinct leaves: the singular leaf at the origin and two cones\footnote{We continue to refer to these as ``cones" despite the exclusion of their tips.}, one with $T>0$ and the other with $T<0$, both of which are regular leaves. For $c>0$, we obtain a one-sheeted hyperboloid, which is topologically equivalent to a cylinder. In contrast, for $c<0$, we observe two-sheeted hyperboloids, representing two separate regular leaves, each with a trivial topology. An illustration of these configurations is provided in Fig.~!\ref{fig1}.
			\begin{figure}[!]
				\subfloat[a][]
				\centering
				\scalebox{0.1}[0.1]{\includegraphics{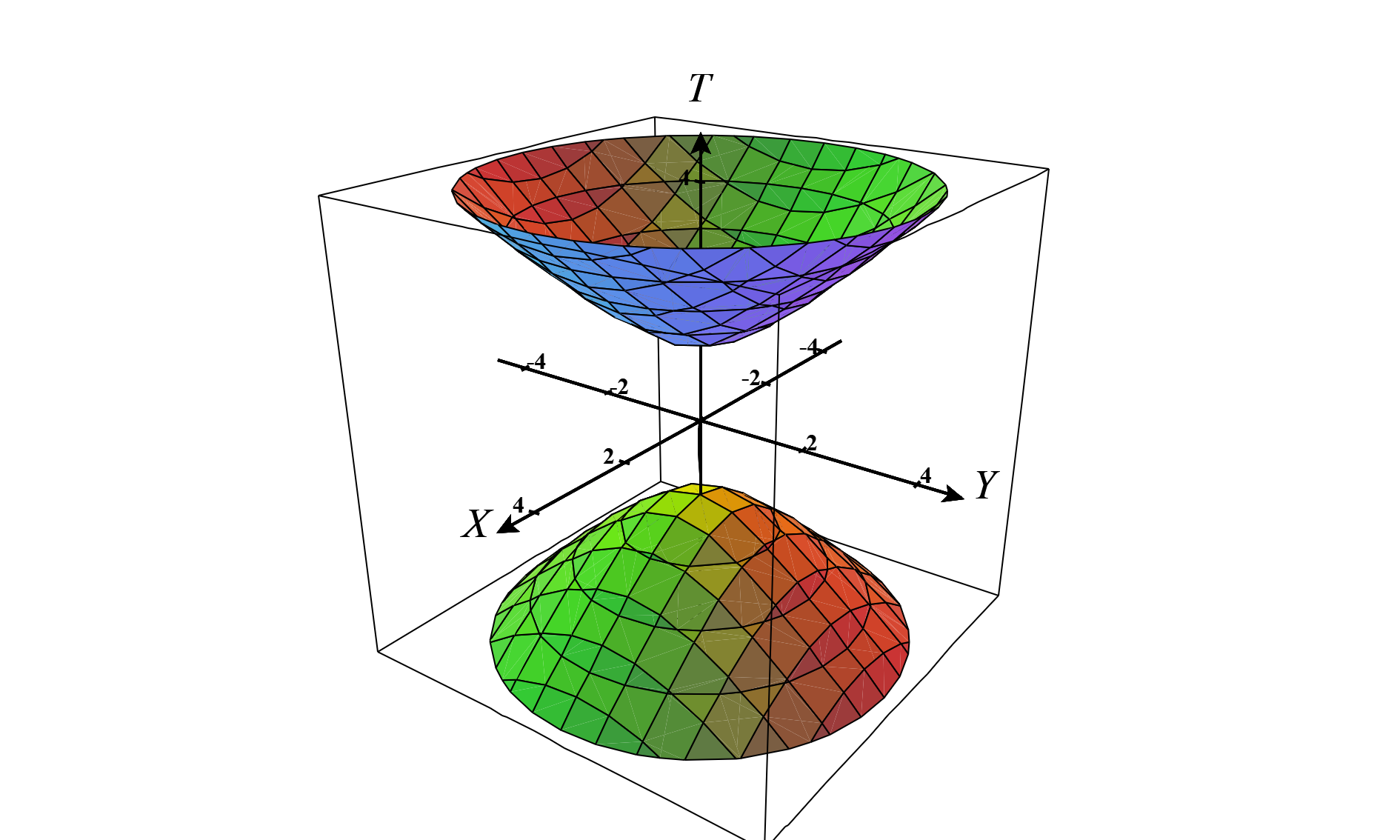}} 
				\hspace{-2.3cm}
				\subfloat[b][]
				{\includegraphics[scale=0.1]{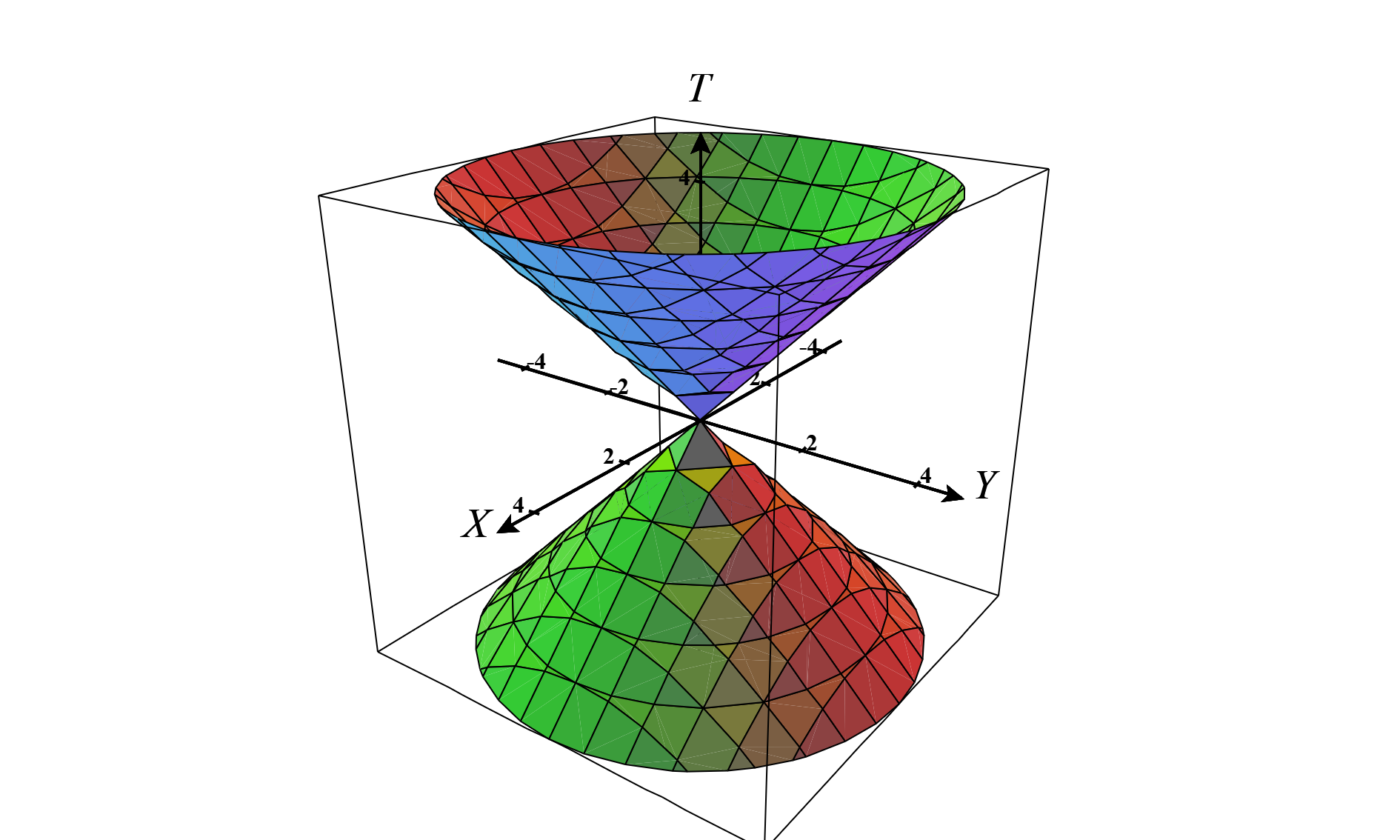}} 
				\hspace{-2.3cm}
				\subfloat[c][]
				{\includegraphics[scale=0.1]{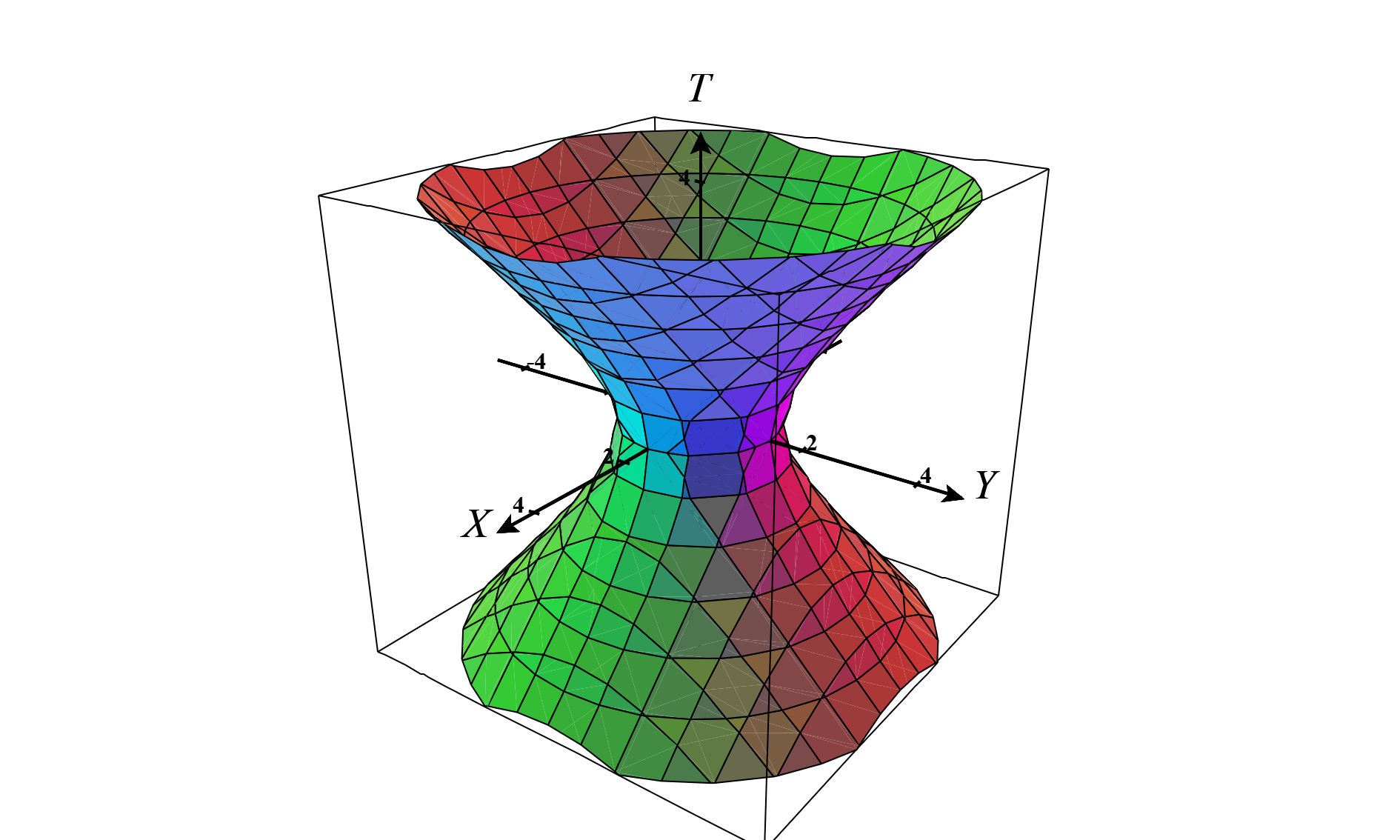}}
				\caption{Symplectic leaves on $\sl(2)^*$:  (A) for $c=-1$, (B) for $c=0$, and (C) for $c=1$  \\ } \label{fig1}
			\end{figure}
		\end{exmp}

		Among the special non-linear cases of equation \eqref{sl2nonlin}, let us first examine the following: 
		
		\begin{exmp}[Quadratic brackets] \label{ex:quad}
			With the choice 
			\begin{equation} \label{Uqua}
				U_{\mathrm{qua}}(z) := {3z^2 -1} \quad , \qquad  V_{\mathrm{qua}}(z) := 0 \, ,
			\end{equation}
			we obtain quadratic brackets from \eqref{sl2nonlin}. The Casimir \eqref{C} takes the form \begin{equation}
				\mathcal C_{\mathrm{qua}} = xy +  z^3 - z \, ,
			\end{equation}
			
			\noindent when choosing $P=0$ and $Q=z^3 - z$. (Changing the integration constants for $P$ and $Q$, leads to the (irrelevant) redefinition  $\mathcal C_{\mathrm{qua}} \mapsto\, e^a\, \mathcal C_{\mathrm{qua}} + b$ for some $a, b\in \R$).
			
			\smallskip \noindent In this case, there are precisely two singular leaves, located at $(0,0,\tfrac{1}{\sqrt{3}})$ and $(0,0,-\tfrac{1}{\sqrt{3}})$, as $\pm \tfrac{1}{\sqrt{3}}$ are the two zeros of the function $U$, cf.\ \eqref{Uqua}. These singular leaves correspond to the critical values $c=\pm 2 \sqrt{3}$ of the Casimir function, implying  that the level set splits into regular and singular leaves for those values. For all other values of $c$ we have two-dimensional, and thus regular symplectic leaves. For example, if we choose $c=1$, we obtain a topologically trivial symplectic leaf, depicted in Fig.~\!\ref{fig:cquad1}. 
			\begin{figure}[H]
				\centering
				\includegraphics[width=0.4\linewidth]{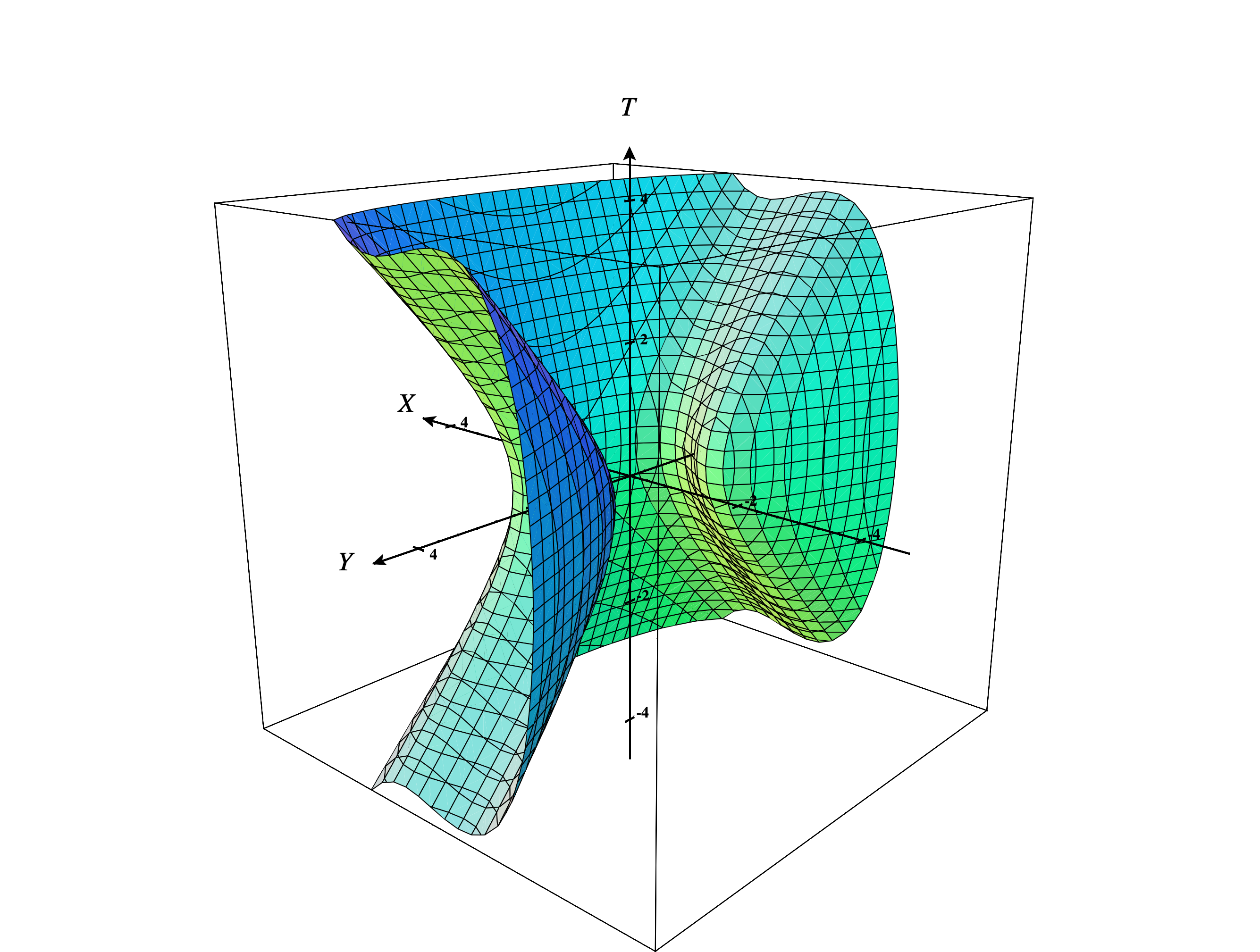}
				\caption{Quadratic bracket: symplectic leaf for $c=1$}
				\label{fig:cquad1}
			\end{figure}
			Conversely, for $c=0$, we encounter a more complex symplectic leaf's structure. This leaf essentially consists of two cylinders oriented orthogonally to each other and joined to form a single leaf with two holes—see Fig.~!\ref{fig:cquad-2}. Topologically, this surface $S$ is equivalent to a genus-one surface (a torus) with one puncture.
			\begin{figure}[H]
				\subfloat[a][]
				\centering
				\scalebox{0.1}[0.1]{\includegraphics{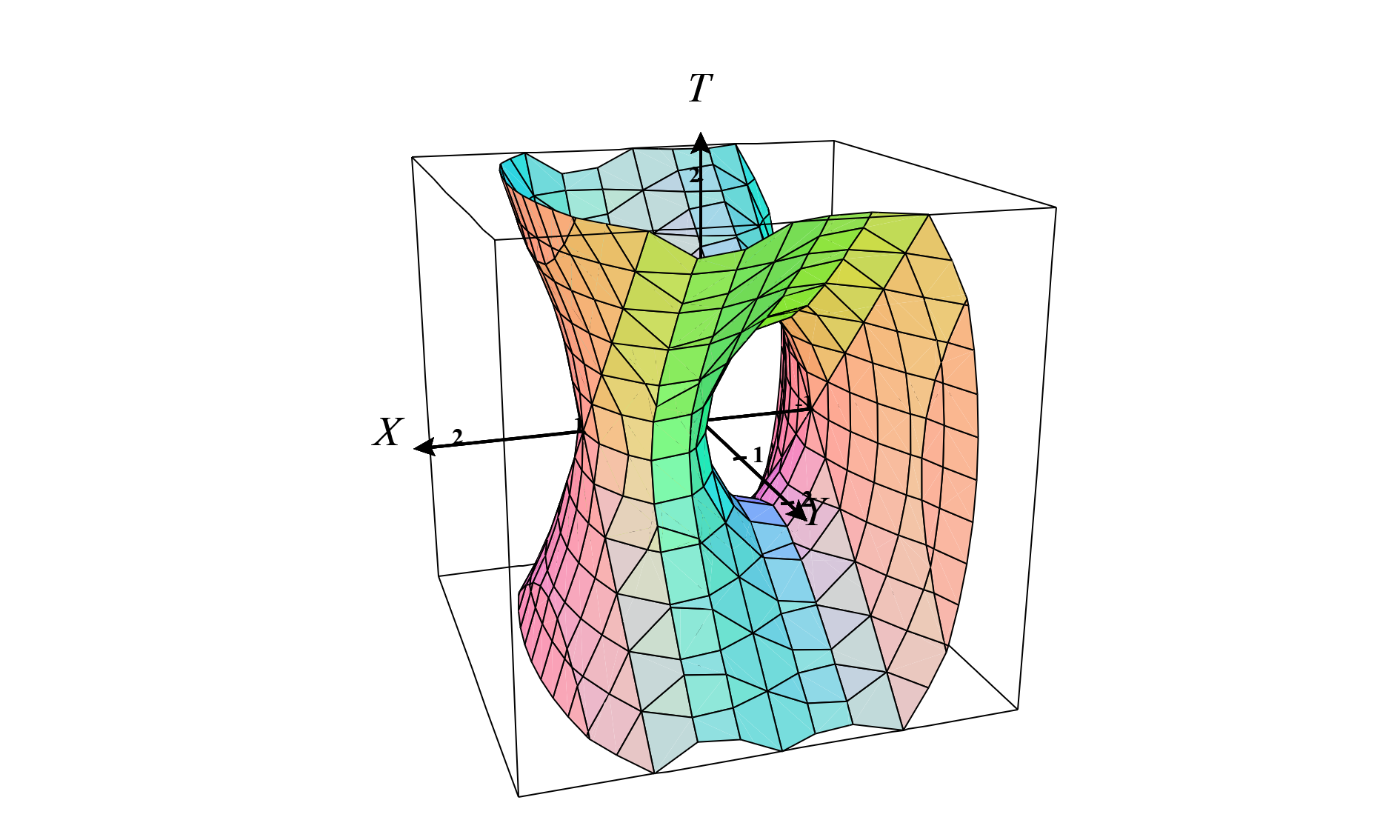}} 
				\subfloat[c][]
				{\includegraphics[scale=0.1]{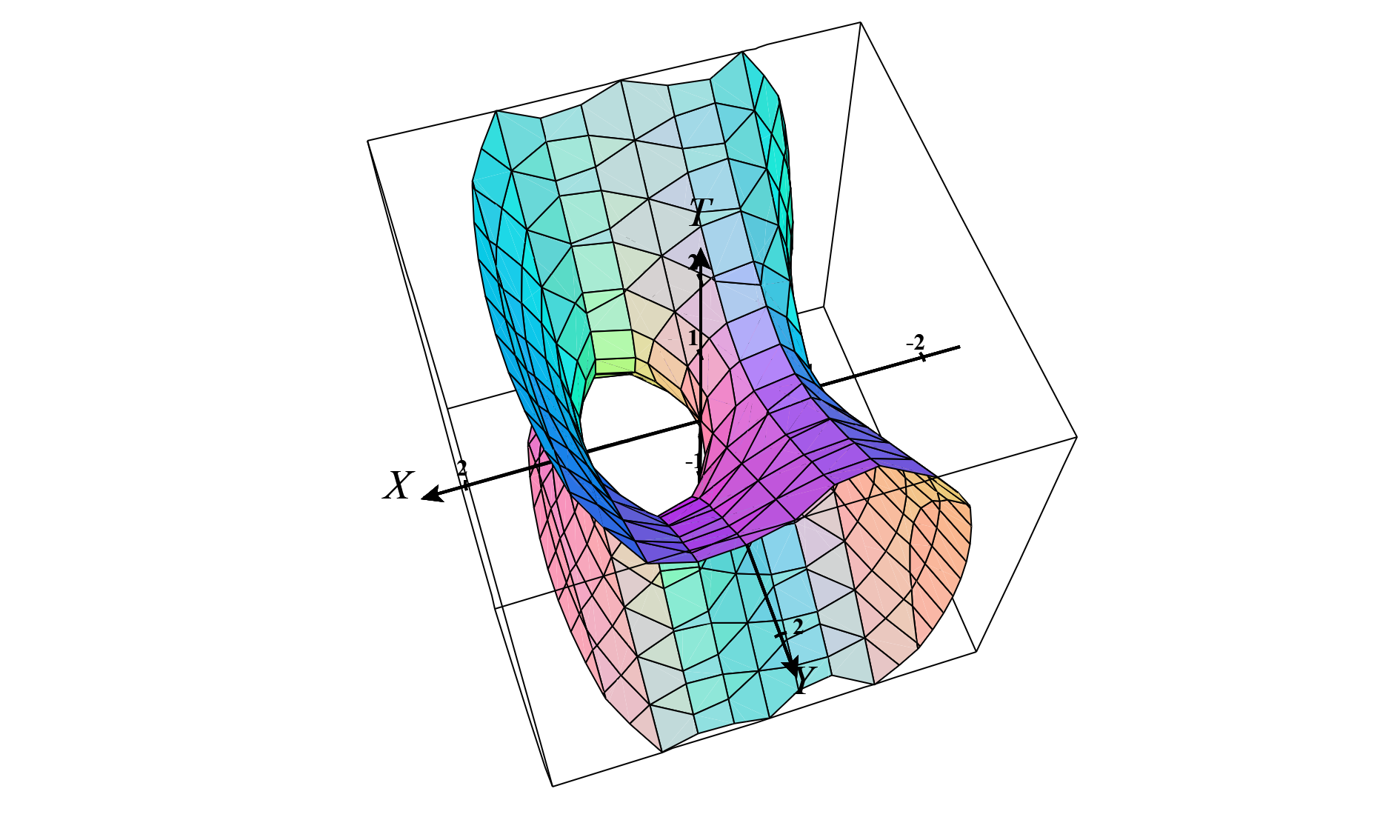}}
				\caption{Quadratic bracket: symplectic leaf $S$ for $c=0$:  (A) seen from the side (B) seen from above. Topologically $S$ is a punctured torus. } \label{fig:cquad-2}
			\end{figure}
		\end{exmp}
		
		The next example is a Poisson-Lie group based on the three-dimensional \emph{book Lie algebra} \footnote{The regular coadjoint orbits of this Lie algebra resemble the pages of an open book, hence its name.} \cite{LA}. This example illustrates a case where the function $V$ in equation \eqref{sl2nonlin} is non-zero, in contrast to the previous examples.
		
		\medskip
		\begin{exmp}[Poisson-Lie group \cite{BMR}\footnote{For an easier comparison with  \cite{BMR}, use the coordinates  $x_1:=z$, $x_2 := \sqrt{2}\, y$, and  $x_3 := \sqrt{2}\, x$ in the formulas below and rescale $C_{\mathrm{grp}}$ by a factor of $\tfrac{1}{2}$.}] \label{PLG}
			The book Lie algebra $\widetilde{ \mathfrak{g} }$ is defined by the Lie brackets: $[\widetilde z, \widetilde x]= -\eta \widetilde x$, $[\widetilde z, \widetilde y]= -\eta \widetilde y$, and $[\widetilde x, \widetilde y]= 0$. Here, $\widetilde{x}$, $\widetilde{y}$, and $\widetilde{z}$ are generators of the algebra, and $\eta$ is a non-zero real parameter whose significance will become visible shortly.  The pair $(\widetilde{\mathfrak{g}}, \sl(2,\mathbb{R}))$ forms a Lie bialgebra. Integrating $\widetilde{ \mathfrak{g} }$ to its unique connected and simply connected Lie group $\cal G$ thus leads to a Poisson-Lie group. It turns out that, as a manifold, 
			${\cal G} \cong \R^3$, and its Poisson structure is given by \eqref{sl2nonlin}  for the choice \begin{equation}
				U_{\mathrm{grp}}(z) := \frac{1-e^{-2\eta z}}{2\eta} \quad , \qquad V_{\mathrm{grp}}(z) := \eta. \label{grp}
			\end{equation}
			We now see that $\eta$ can be considered as a deformation parameter here: in the limit of sending $\eta$ to zero, we get back the linear brackets \eqref{sl2brackets}. For an appropriate choice of integration constants, \eqref{C} yields 
			\begin{equation}
				\mathcal C_{\mathrm{grp}} = xy e^{\eta z} + \frac{\cosh(\eta z)-1}{\eta^2}
			\end{equation}
			which also reduces to the $\sl(2,\mathbb{R})$-Casimir \eqref{Clin} in the limit. This example provides a particularly interesting 1-parameter family of deformations of the linear Poisson structure on $\sl(2,\mathbb{R})^*$ within the infinite-dimensional deformation space governed by the two functions $U$ and $V$ as it corresponds to a Poisson-Lie group. Three symplectic leaves of this Poisson-Lie group are depicted in Figure~\ref{figeta}. 
			\begin{figure}[H]
				\subfloat[a][]
				\centering
				\scalebox{0.1}[0.1]{\includegraphics{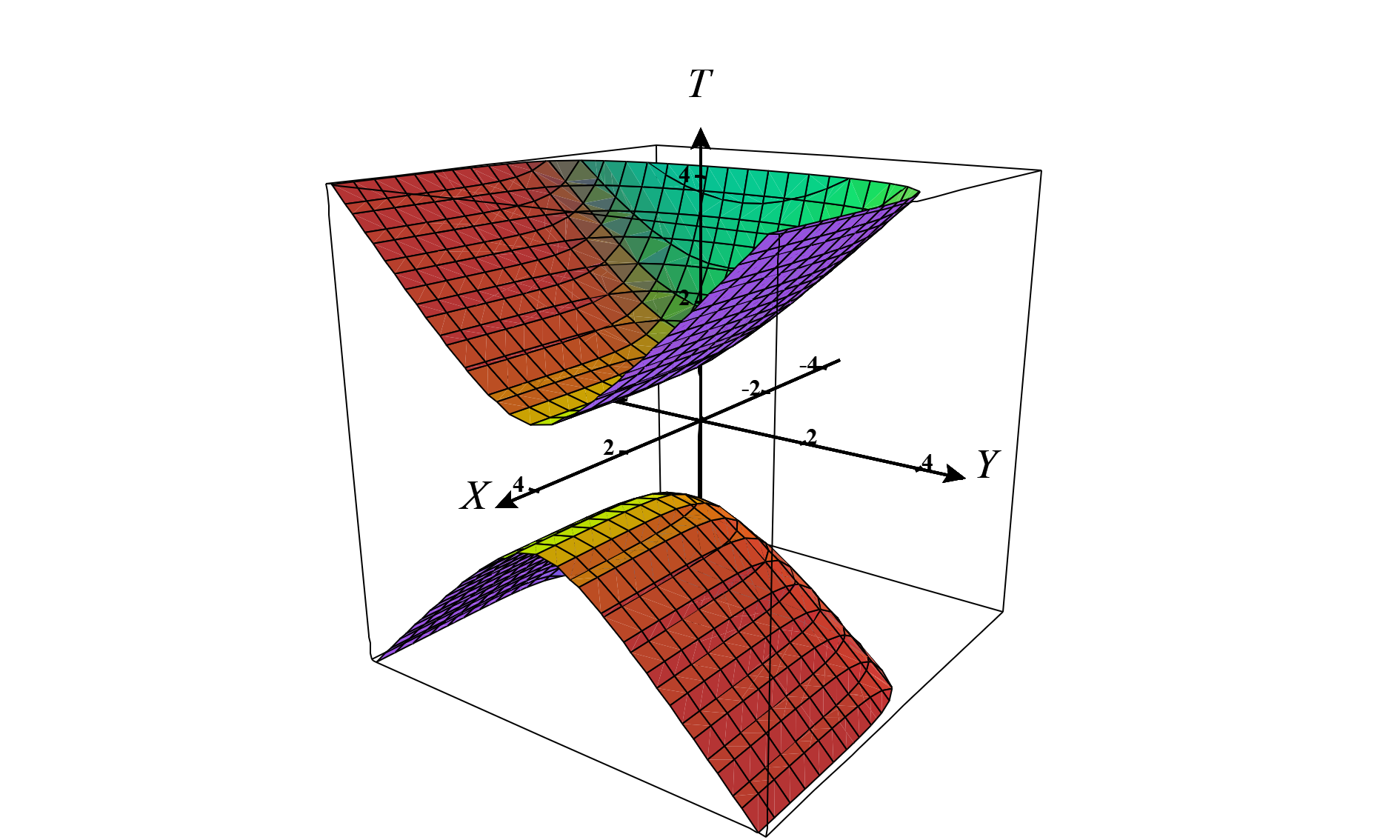}} 
				\hspace{-2.3cm}
				\subfloat[b][]
				{\includegraphics[scale=0.1]{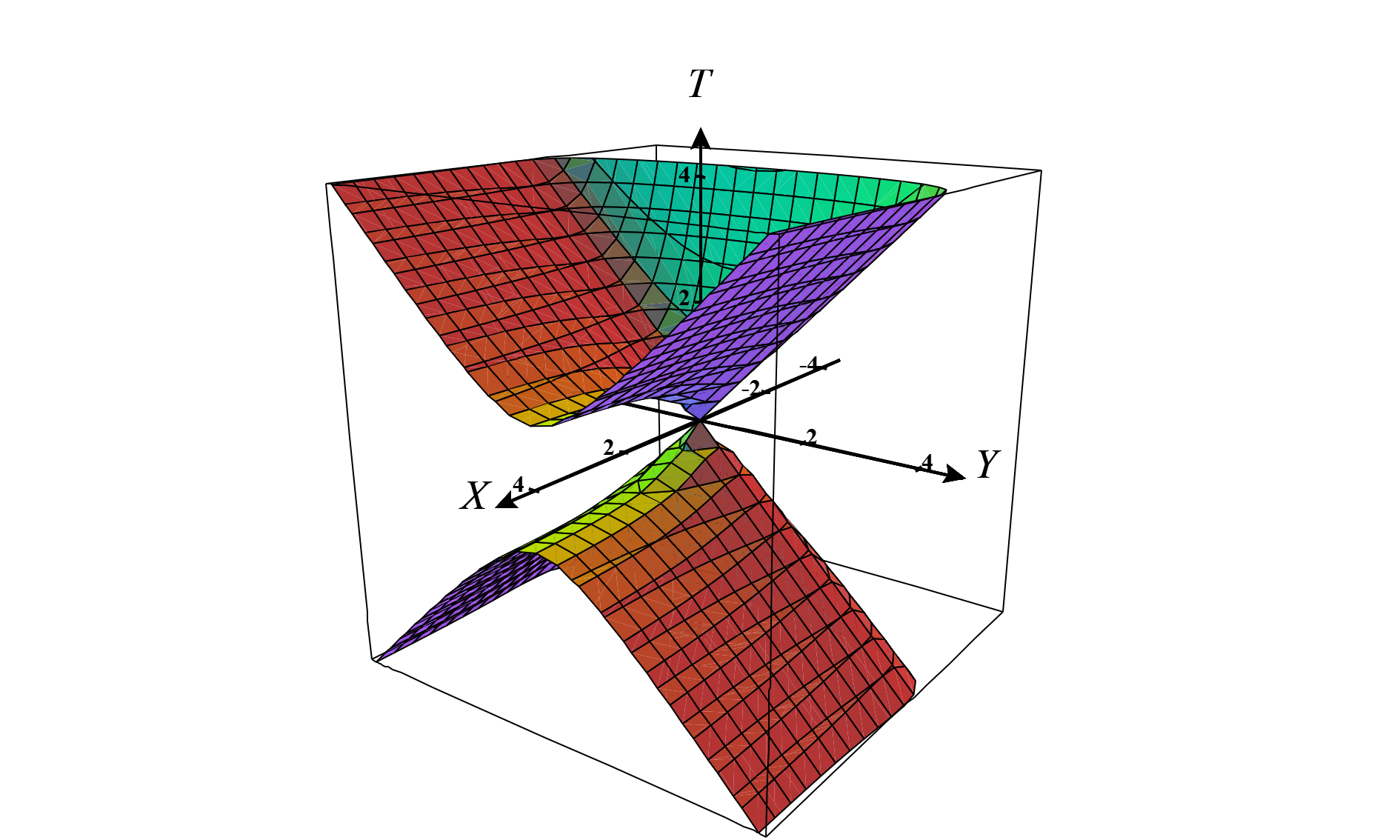}} 
				\hspace{-2.3cm}
				\subfloat[c][]
				{\includegraphics[scale=0.1]{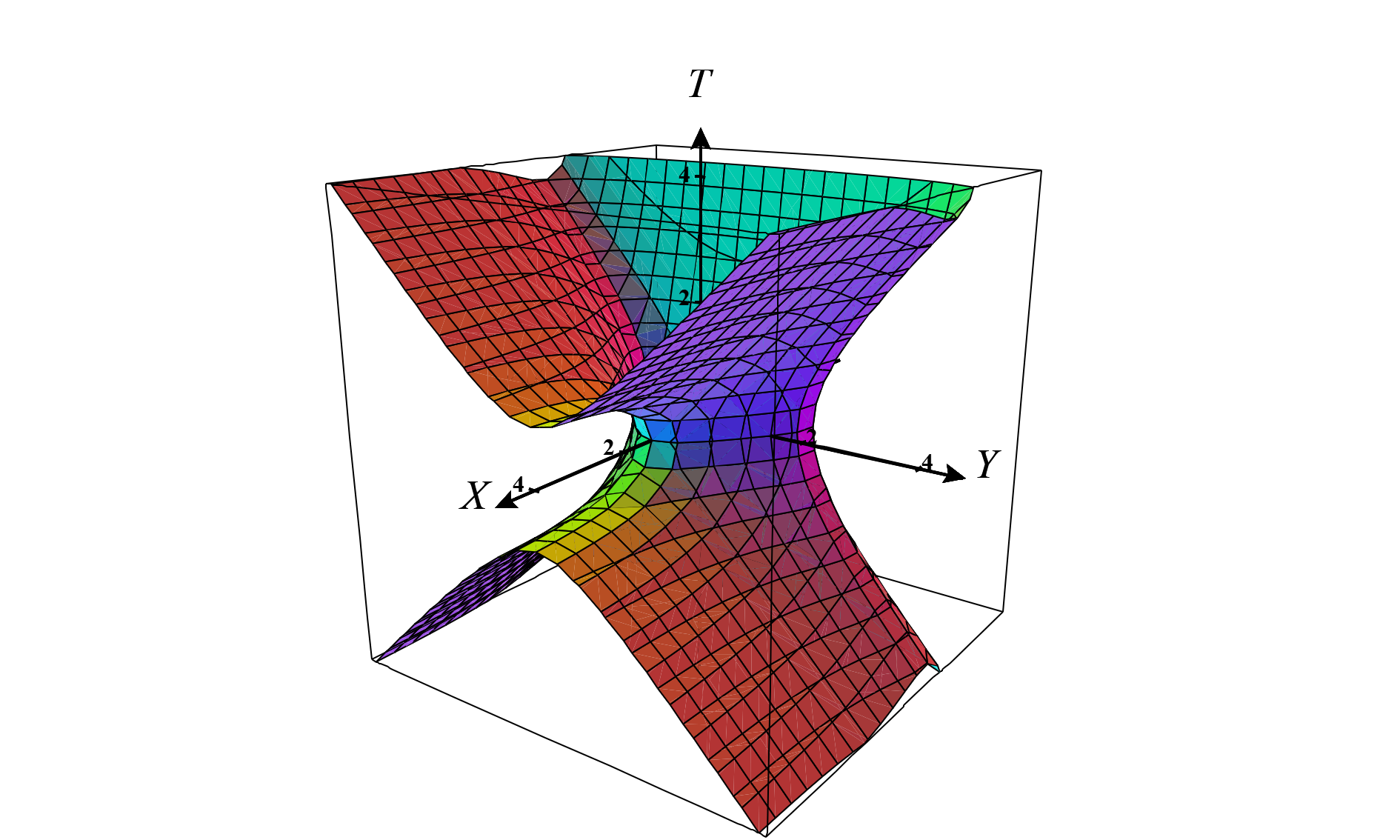}}
				\caption{Symplectic leaves on Poisson-Lie group for the deformation parameter $\eta =1$:  (A) for $c=-1$, (B) for $c=0$, and (C) for $c=1$. For $\eta \to 0$, they more and more approach the leaves shown in Fig.~\!\ref{fig1}.} \label{figeta}
			\end{figure}
		\end{exmp}

		\section{Topological nature of certain symplectic leaves}\label{sec:5}
		In this section, we discuss the topological nature of the symplectic leaves for the class of Poisson structures (\ref{sl2nonlin}). Specifically, we investigate the topological characteristics of the symplectic leaves for our two chosen examples. To better understand these characteristics for various choices of $U$ and $c$, while keeping $V$ at zero for simplicity, it is helpful to consider the following function:
		\begin{rem}
			Consider the Casimir function \ref{C} of the bracket \ref{sl2nonlin}. The number and type of zeros of the function 
			\begin{equation} \label{h}
				h_c(z) := Q(z) - c \,,
			\end{equation}
			determine the topology of the symplectic leaves.
		\end{rem}	
		\medskip  \noindent 
		\begin{exmp}
			
			Let us consider the Poisson bracket in the Example \ref{sl2leaves}. We depict the corresponding function $h_c$ for the three values $c=-1$, $c=0$, and $c=+1$ in Fig.~\!\ref{fig:hlin}A.  
			\begin{figure}[h]
				\subfloat[]
				\centering
				\scalebox{0.2}[0.2]{\includegraphics{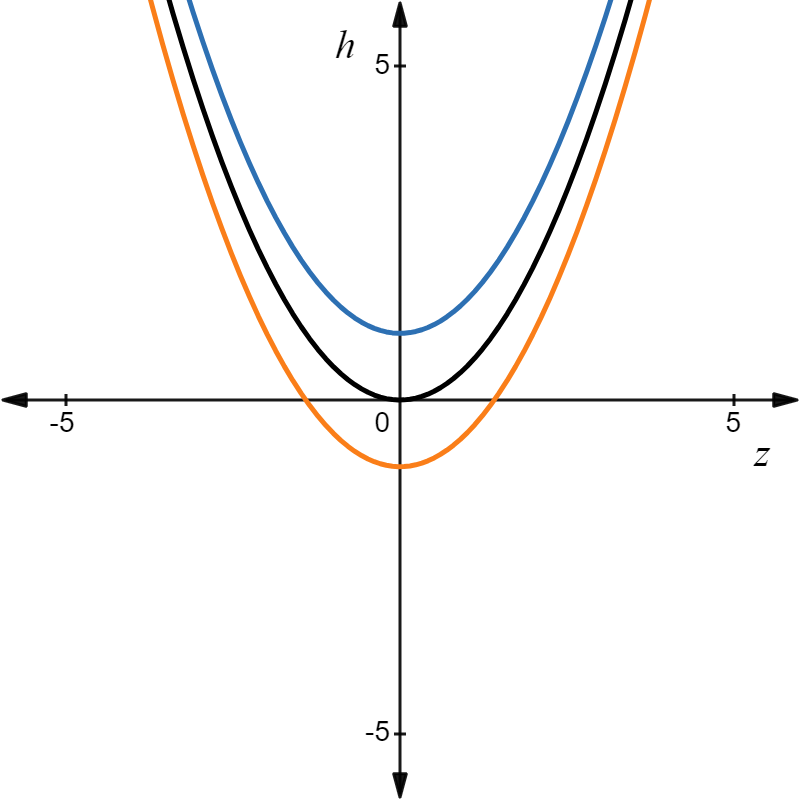}}
				\hspace{2cm}
				\subfloat[][]
				{\includegraphics[scale=0.2]{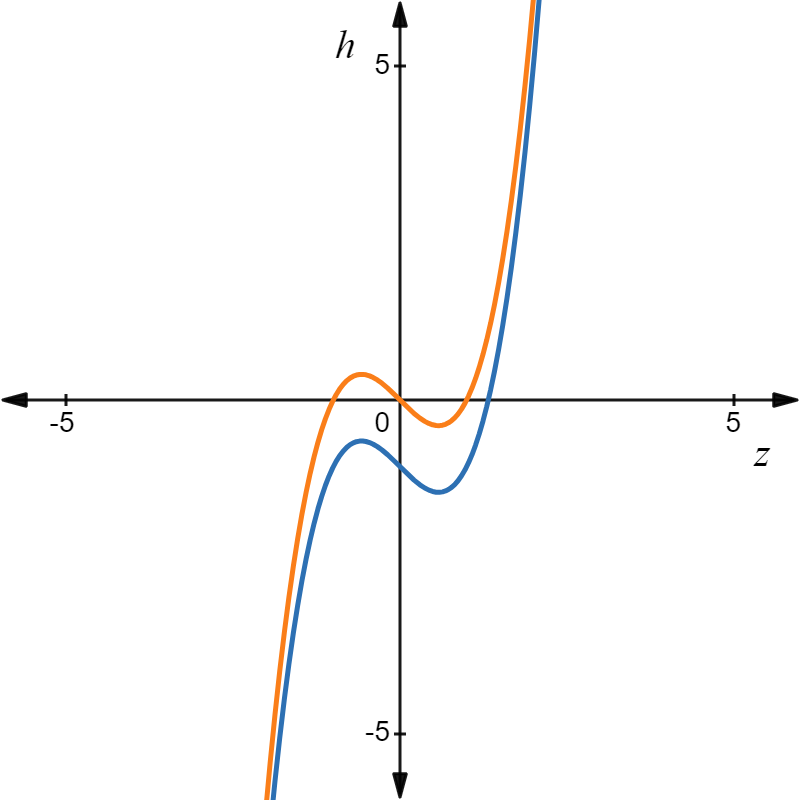}}
				\caption{(A) Function $h_c$ for the linear Poisson structure on $\sl^*_2$, orange color for $h_{-1}$, black color for $h_{0}$ and blue color for $h_{1}$. (B) Function $h_c$ for the quadratic Poisson structure, orange color for $h_{0}$ and blue color for $h_{1}$. } \label{fig:hlin}
			\end{figure}
			We observe that $h_{-1}$ has no zeros, and the corresponding leaf is topologically trivial, as shown in Fig.~\ref{fig1}A. In contrast, $h_{1}$ has two simple roots, and the leaf $S$ for $c=1$ is topologically a cylinder. This qualitative behavior of $h_c$ and the corresponding leaves remains consistent as long as the sign of $c$ does not change.
			However, for the special value $c=0$, we encounter a unique situation: the function has one multiple root, which corresponds to a singular leaf coexisting with two regular leaves.
			
			In fact, the latter observation is not a coincidence. Recall that all singular leaves lie on the $z$- or $X$-axis at points where the function $U$ vanishes, as shown in equation \eqref{Lsing}. From the definition in \eqref{h}, we can deduce that for every value of $z$ where $U$ vanishes, $h_c'$ also vanishes. Consequently, singular leaves appear for values of $c$ where $h_c$ has a multiple zero.
		\end{exmp}	
		\begin{exmp}
			Let us now examine the second example, Example \ref{ex:quad}, and depict the graphs of $h_1$ and $h_0$—see Fig.~\ref{fig:hlin}B.
			For $c=1$, the function has one simple zero, and the corresponding leaf is topologically equivalent to a plane.
			For $c=0$, the function has three simple zeros. In this case, the corresponding leaf $S$ has a fundamental group $\pi_1(S) = \mathbb{F}_2$, which is the free group with two generators. The general situation can be summarized as follows:
			
			In general Classification, if the function $h_c$ has $n \geq 1$ simple zeros, then:
			\begin{itemize}
				\item The symplectic leaf $S$ is a Riemann surface of genus $\left[\tfrac{n+1}{2}\right]-1$, where $[\cdot]$ denotes the integer part.
				\item If $n$ is odd, $S$ has one puncture.
				\item If $n$ is even, $S$ has two punctures (boundary components).
			\end{itemize}
			Specifically, for instance: 
			\begin{itemize}
				\item For $n=3$: $S$ is a punctured torus, consistent with the previously mentioned fundamental group $\mathbb{F}_2$.
				\item For $n=4$: $S$ is a torus with two punctures (see Fig.~\ref{fig:fourzeros}).
				\item For $n=5$: $S$ is a genus two surface with one puncture (see Fig.~\ref{fig:fivezeros}).
			\end{itemize}
			
			\begin{figure}[h]
				\subfloat[a][]
				\centering
				\scalebox{0.1}[0.1]{\includegraphics{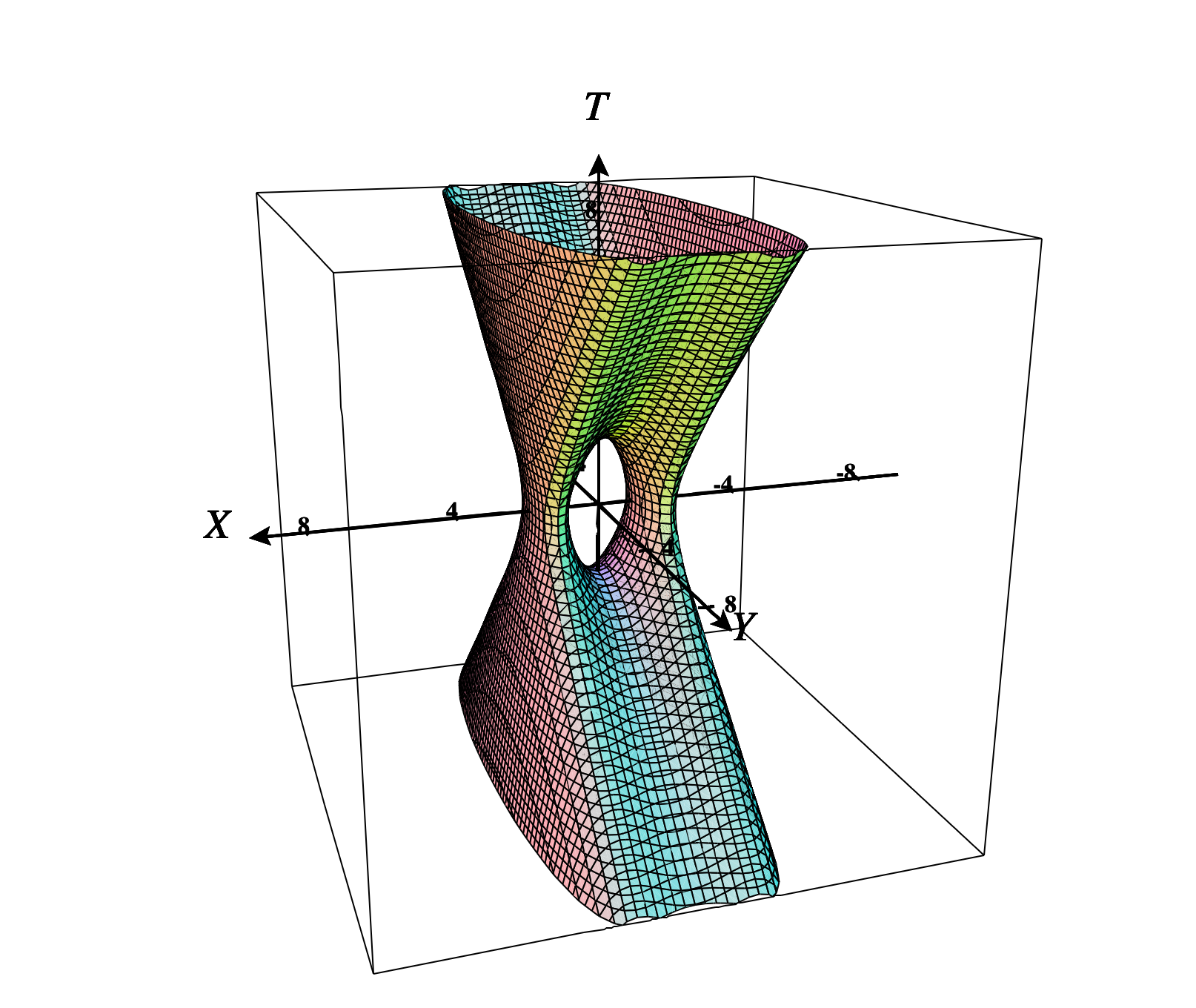}} 
				\subfloat[c][]
				{\includegraphics[scale=0.1]{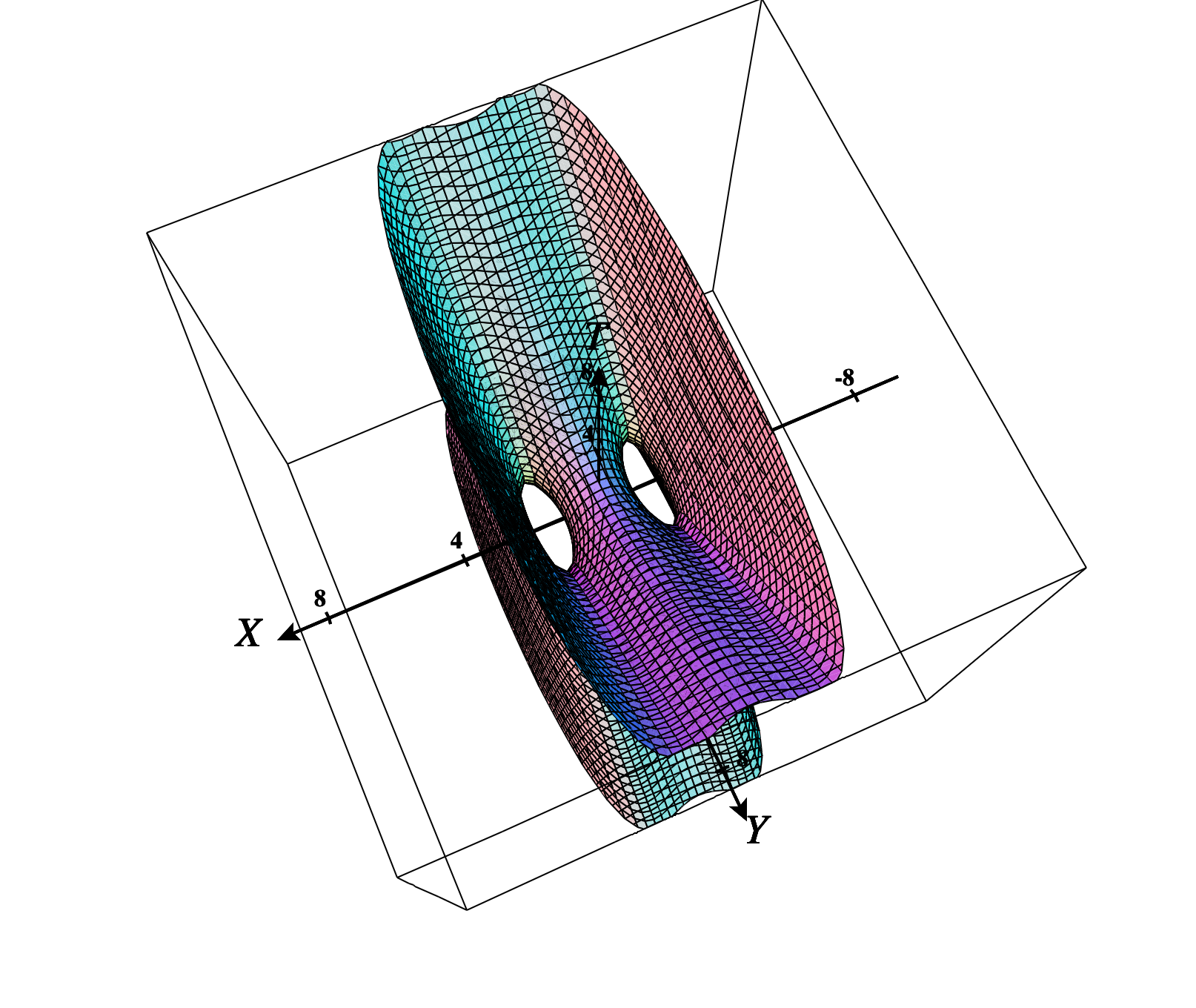}}
				\caption{The symplectic leaf when  $h(z)=2(z-2)(z-1)(z+1)(z+2)$, in (A) from the side and in (B) from above. Topologically it is a genus one surface with two punctures or boundary components.} \label{fig:fourzeros}
			\end{figure} 
			\begin{figure}[h]
				\subfloat[a][]
				\centering
				\scalebox{0.1}[0.1]{\includegraphics{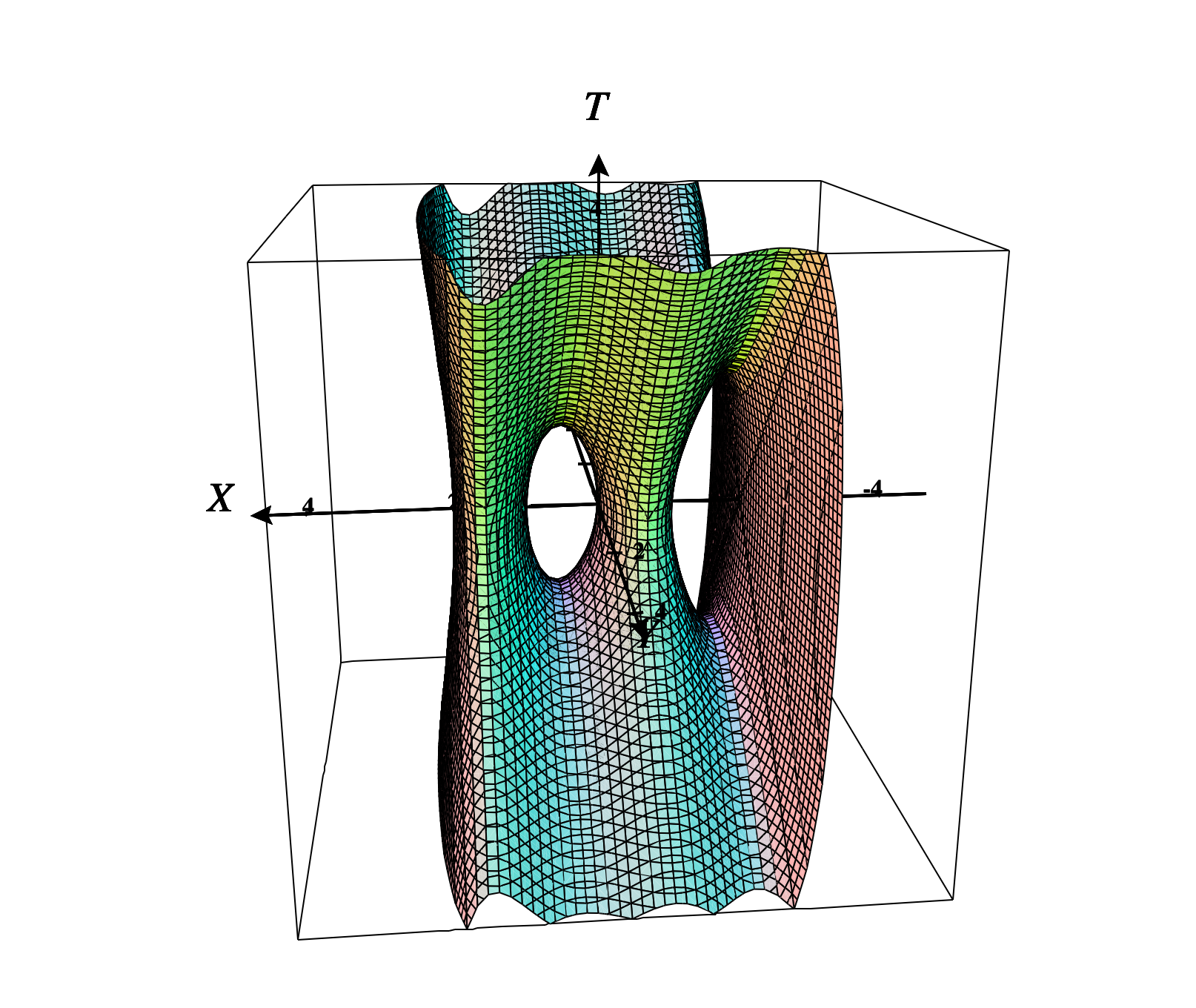}} 
				\subfloat[c][]
				{\includegraphics[scale=0.1]{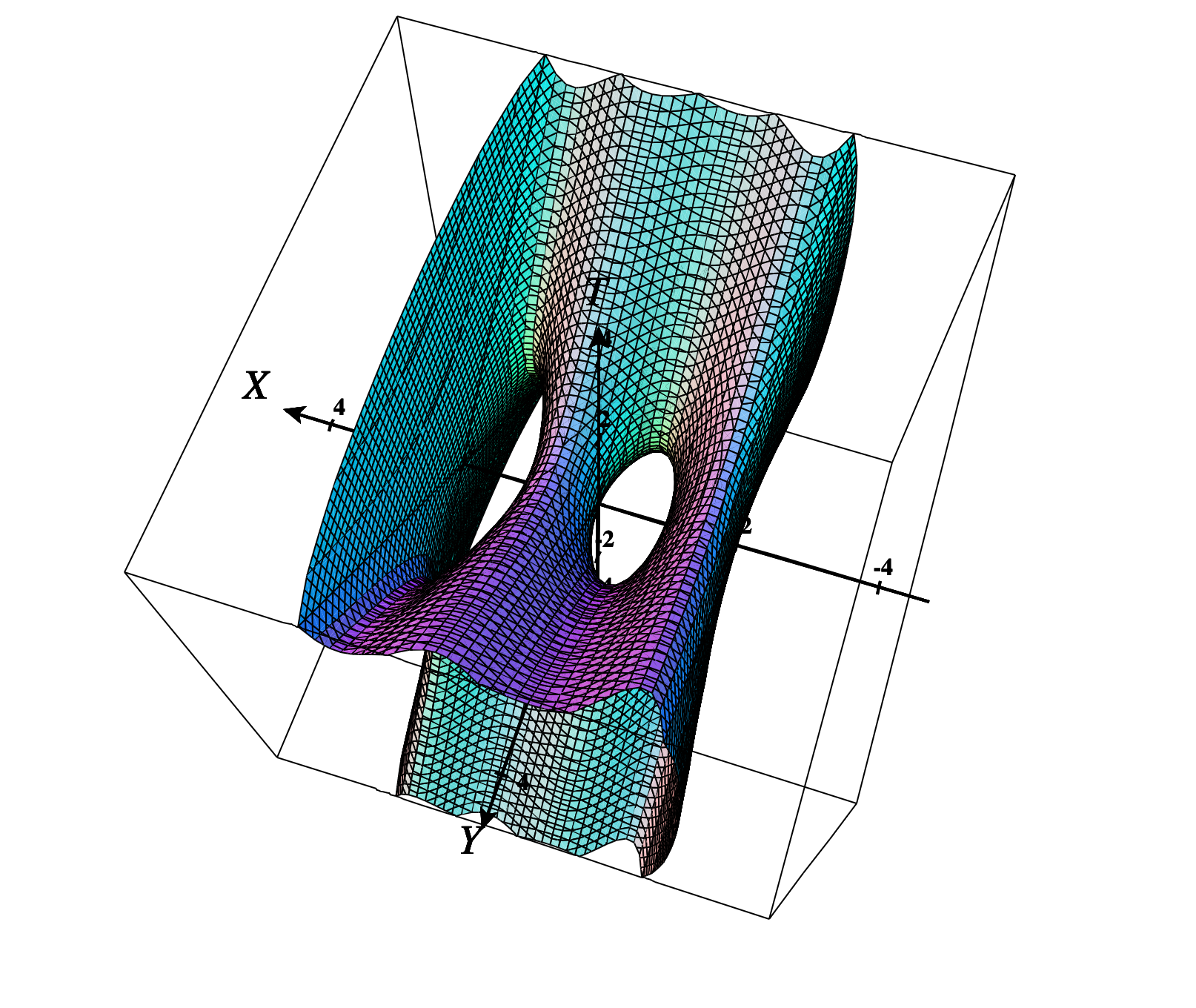}}
				\caption{The symplectic leaf when  $h(z)=2(z-2)(z-1)z(z+1)(z+2)$, in (A) from the side and in (B) from above. Topologically it is a punctured genus two surface.} \label{fig:fivezeros}
			\end{figure} 
			
			\noindent These leaves can be obtained using specific polynomials, for instance:
			\begin{itemize}
				\item For $n=4$: $Q(z)=2(z-2)(z-1)(z+1)(z+2)$
				\item For $n=5$: $Q(z)=2(z-2)(z-1)z(z+1)(z+2)$
			\end{itemize}
			In both cases, we set $c=0$ which yields $h_0=Q$.
			By reverse engineering, the Poisson brackets yielding such leaves can be obtained from the brackets \eqref{sl2nonlin} when choosing $U=Q'$ and $V=0$. For completeness, we mention two special cases:
			\begin{itemize}
				\item If $h_c$ has no zeros: One obtains two planar symplectic leaves.
				\item If $h_c$ has zeros of some multiplicity: The leaf contains singular components.
			\end{itemize}
			This classification demonstrates how the topology of symplectic leaves is intricately linked to the properties of the defining function $h_c$.
		\end{exmp}
		Evidently, the freedom in choosing the function $U$ in the brackets \eqref{sl2nonlin} permits one to obtain much more intricate symplectic leaves than in the linear case shown in Figure \ref{fig1}. This flexibility allows for the creation of a wide variety of topological structures leading to the following conjecture:
		\begin{Conjecture}
			For any integer $k \in \mathbb{N}$, by an appropriate choice of $U$, while $V=0$, we can create a sample of symplectic leaves which topologically are Riemann surfaces of genus $k$—with optionally one or two boundary components.
		\end{Conjecture}
		\noindent This underscores the rich topological diversity achievable through careful selection of the function $U$ in the Poisson bracket structure. It demonstrates that the complexity of symplectic leaves can be tailored to produce Riemann surfaces of arbitrary genus, with the additional flexibility of including boundary components as desired.
		
		\section{Degeneracy structure of the metriplectic tensor $\M$: red zones
		} \label{sec:6}
		\medskip
		In this section, we consider the class of Poisson structures given by (\ref{sl2nonlin}) and choose a metric on the ambient space in which the leaves are embedded, namely the Poisson manifold $\mathbb{R}^3$. We then discuss the degeneracy structure of the metriplectic tensor. For the linear $\sl_2^*$ brackets, a natural candidate for this metric is the Killing form of the Lie algebra $\sl_2$, denoted as $\kappa$. The Killing form $\kappa$ has the following properties:
		It is an indefinite metric and its indefinite signature corresponds to the non-compact nature of the Lie group $\SL(2,\R)$. After a trivial rescaling, $g := \tfrac{1}{2}\kappa$, the metric $g$ is defined as
		\begin{equation} \label{g}
			g = 2 \, \d x   \d y +   \d z^2 = -\d T^2 + \d X ^2+   \d Y^2 \, , \end{equation}
		written in both of the coordinate systems used before. Recall that we omit writing the symmetrized tensor product; correspondingly, $\d z^2$ stands for  $\d z \otimes  \d z$. 
		
		\smallskip
		As stated in the introduction, our approach to selecting a metric on the given Poisson manifold is guided by simplicity rather than sophistication. We avoid complex options, such as metrics that would adapt to the symplectic leaves of the given Poisson structure. Consequently, we maintain the metric \eqref{g} unchanged while modifying the Poisson brackets. This approach allows us to focus on the effects of changing the Poisson structure while maintaining a consistent reference frame for our geometric analysis.
		
		\smallskip\noindent 
		The metric \eqref{g} identifies the given $\R^3$ with a 2+1 dimensional Minkowski space, where $T$ serves as the time coordinate, and $X$ and $Y$ are spatial coordinates. Determining where the induced metric $g^S$ on a symplectic leaf becomes degenerate can be challenging and may require a case-by-case analysis. This process can be complex for several reasons:
		\begin{itemize}
			\item An atlas may be needed to cover the leaf, even for two-dimensional leaves.
			\item The degeneracy of $g^S$ in a particular coordinate system might be due to:
			
			a) An intrinsic degeneracy of the metric on the leaf,
			
			b) A poor choice of coordinates (like when writing the standard metric on $\R^2$ in polar coordinates).
		\end{itemize}
		Fortunately, a comprehensive analysis of degeneracy for every possible coordinate system is not necessary. This simplification allows for a more efficient approach to studying the geometric properties of symplectic leaves without the need for exhaustive coordinate-dependent calculations.
		\smallskip 
		
		In the following, we discuss the degeneracy structure of the metriplectic tensor for the class of Poisson structures given in equation (\ref{sl2nonlin}) and the metric defined in equation (\ref{g}). By Definition \ref{metricplectic}, the matrix corresponding to the metriplectic tensor $\M$ in the basis $\partial_x \otimes \partial_x$, $\partial_x \otimes \partial_y$, $\ldots$, $\partial_z \otimes \partial_z$ takes the form
		\begin{equation} \label{Mmatrix}
			\left[\M\right]  = \begin{pmatrix}
				x^2 & -xy - W^2 & xW \\ -xy - W^2 & y^2 & yW \\ xW& yW &-2xy
			\end{pmatrix} 
		\end{equation}
		where we introduced the abbreviation
		\begin{equation}\label{W}
			W =  U(z)+ xy \,  V(z) \, . 
		\end{equation}
		\begin{prop} \label{PropR}
			Let $(\R^3, g, \Pi)$ be one of the pseudo-Riemmanian Poisson manifolds considered in this section. 
			The points $m = (x,y,z) \in \R^3$ for which the function  
			\begin{equation} \label{f}
				f(x,y,z) := 2xy + W^2(x,y,z) \, 
			\end{equation}
			vanishes belong exclusively to one of the following two classes:
			\begin{itemize}
				\item  $m$ is an $\M$-singular point 
				\item  $\{m\}$ is a singular  symplectic leaf.
			\end{itemize}   
		\end{prop}
		\begin{proof}
			We always have $\Im \sharp_\M \subseteq \Im \Pi^{\sharp}$. Consequently, whenever $\Pi_m$ vanishes—i.e., at the points $m\in {\cal L}_{sing}$—$\M_m$ also vanishes. There are no other points where $\M$ vanishes, as inspection of \eqref{Mmatrix} demonstrates. Since \eqref{f} vanishes at these points as well, we have proven the second item in the proposition.
			
			\noindent 
			It remains to show that all other points $m\in \R^3$ where $f$ vanishes are those where $\mathrm{rk} \, \M_m = 1$. First, let us prove that this condition is sufficient: By replacing $-(xy +W^2)$ with $xy$ in the first two rows and $-2xy$ with $W^2$ in the third row, we observe that each of the three rows is proportional to $(x \quad y \quad W)$. Thus, at points where \eqref{f} vanishes, the rank of $\M$ is one. 
			
			\medskip\noindent  	To show that $f=0$ is also necessary for $\mathrm{rk} \, \M_m = 1$, first note that if both $x$ and $y$ vanish, the rank of $\M$ cannot be odd. Now, suppose $x \neq 0$: Subtract the first row multiplied by $\tfrac{1}{x}W$ from the third row. The new third row becomes $(0 \quad * \quad f)$. Therefore, $f=0$ is necessary for $\M$ to have rank one in this case. A similar argument holds for $y \neq 0$. Alternatively, we can note that $\M$ remains unchanged under the diffeomorphism $(x,y,z) \mapsto (y,x,z)$, leading to the same conclusion.
		\end{proof}
		
		The pseudo-Riemannian Poisson  manifolds $(M,\Pi,g)$ considered in this section are defined by $M=\R^3$, the pseudo-Riemannian metric \eqref{g}, and the bivector (cf.~\eqref{sl2nonlin}) 
		\begin{equation} \label{Pi}
			\Pi = W(x,y,z) \, \partial_x \wedge \partial_y - x \, \partial_x \wedge \partial_z + y \, \partial_y \wedge \partial_z \, , 
		\end{equation}
		where the function $W$ is given in  \eqref{W} and $U$ and $V$ can be chosen arbitrarily.  The Theorem \ref{gradient th}, is applicable whenever we exclude $\M$-singular points, while singular symplectic leaves—all pointlike in our case, see \eqref{Lsing}—are admissible. 
		Thus, in view of Prop.~\ref{PropR}, we have the following definition:
		\begin{rem}     
			The set is given by
			\begin{equation} \label{R}		\mathcal R := \{ (x,y,z) \in \R^3 \vert f(x,y,z)=0\} \backslash {\cal L}_{sing}\, ,
			\end{equation}
			where ${\cal L}_{sing}$ defined in \eqref{Lsing}, characterizes the red zone of the given triple $(M,\Pi,g)$. Notably, $\mathcal R$ is a smooth manifold. 
		\end{rem}
		\noindent	More explicitly, the function $f$ is given by  
		\begin{equation} \label{xy}
			f(x,y,z)= U^2 + 2(1+UV) \, xy + V^2 \, (xy)^2  \, ,
		\end{equation} 
		where $U$ and $V$ depend on $z$ only. 
		
		\medskip Let us illustrate this with the three examples discussed in Section \ref{sec:4}. 
		\begin{exmp}
			For the case $(\sl_2^*,\kappa,\Pi_{\sl_2^*})$ of Example \ref{sl2leaves}, where $U(z)=z$ and $V=0$ in \eqref{Pi}, we see that $f(x,y,z)=2 \mathcal C(x,y,z)$, cf.~\eqref{Clin}. Therefore, the red zone $\mathcal R$, illustrated in Fig.~\!\ref{fig:redzones}A, corresponds precisely to the two regular symplectic leaves obtained from $\mathcal C_{lin}=0$ when the origin is excluded, see Fig.~\!\ref{fig1}B. 
			Consequently, both of these leaves are classified as ``bad" symplectic leaves.
			\begin{figure}[H]
				\subfloat[a][]
				\centering
				\scalebox{0.1}[0.1]{\includegraphics{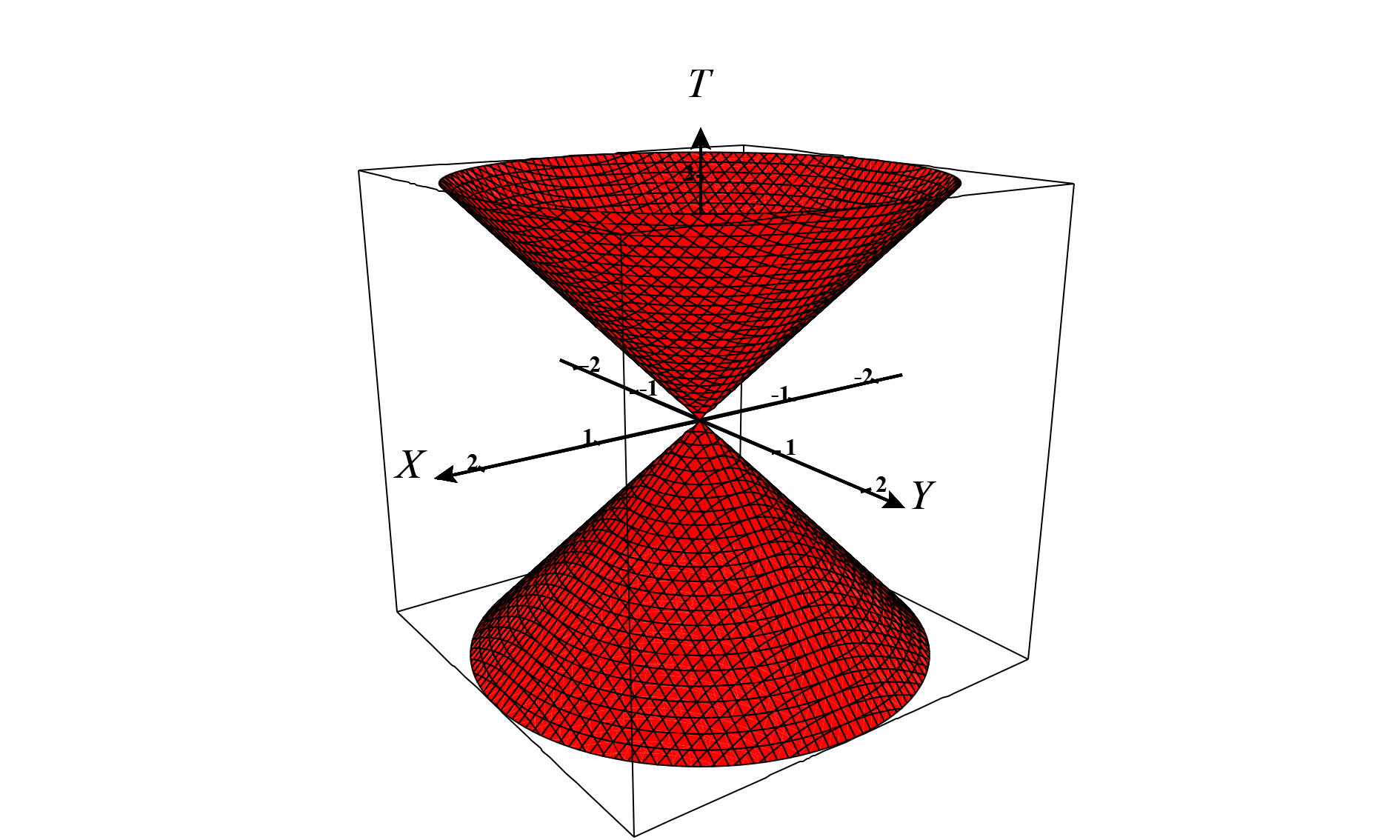}} 
				\hspace{-2.3cm}
				\subfloat[b][]
				{\includegraphics[scale=0.1]{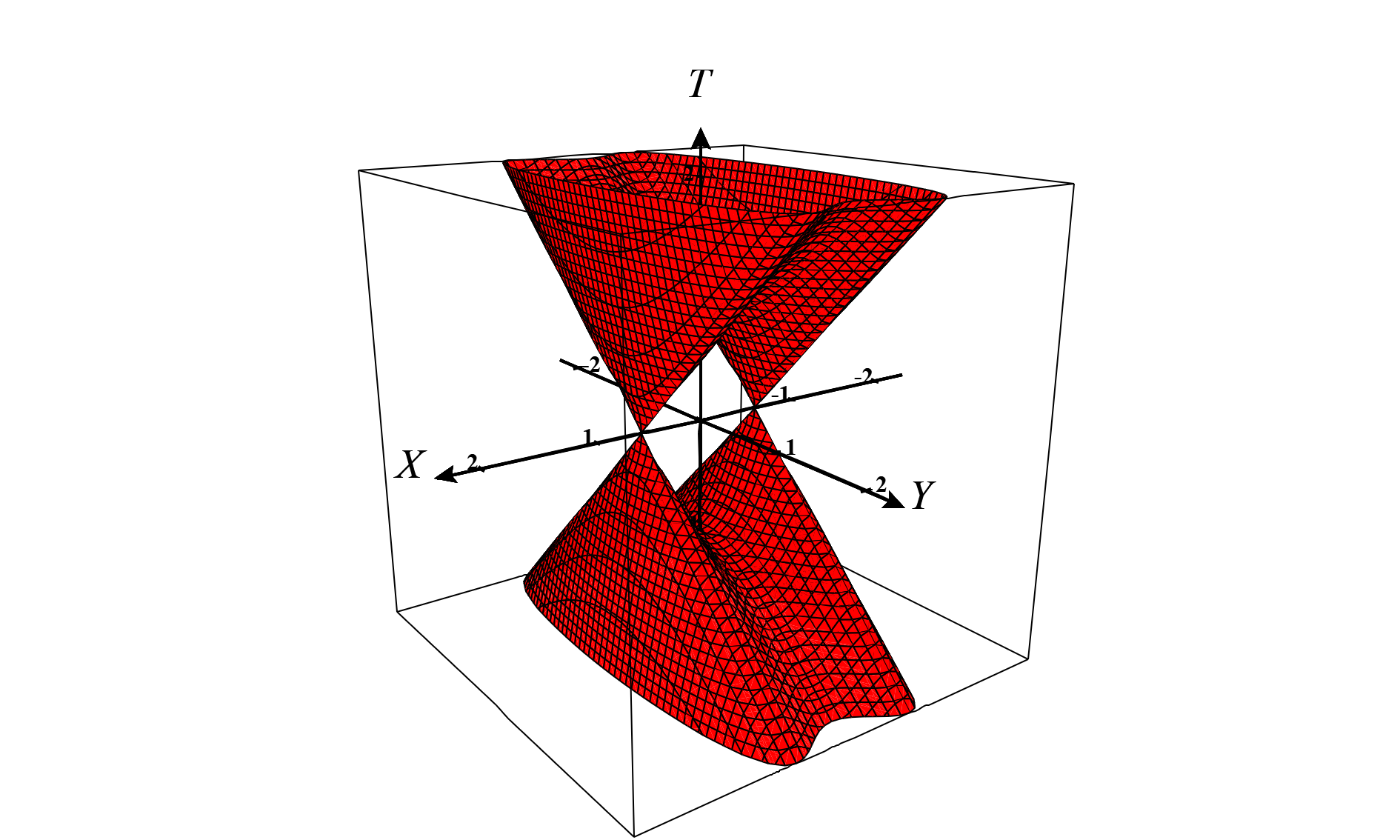}} 
				\hspace{-2.3cm}
				\subfloat[c][]
				{\includegraphics[scale=0.1]{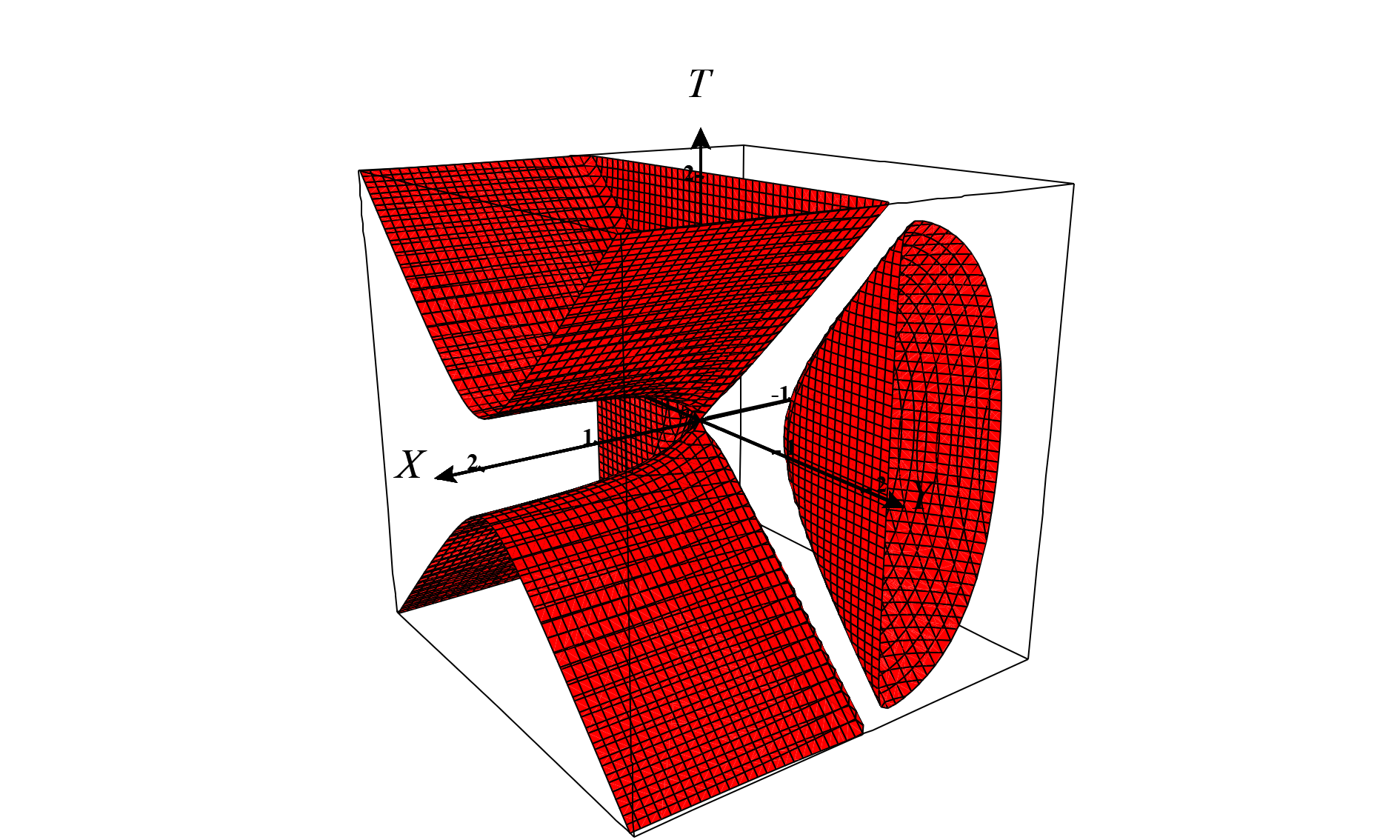}}
				\caption{The red zones $\mathcal R$ for the three main examples, all with respect to the metric $g$ in \eqref{g}: \\ (A) The Lie Poisson manifold $\sl_2^*$, Example \ref{sl2leaves}, \\ (B) Example \ref{ex:quad} with the bivector $\Pi_{qua}$, cf.~\eqref{Piquad}, and \\ (C) the Poisson-Lie group of Example \ref{PLG} for the choice $\eta:=1$.\\
					In all three cases, $\mathcal R$ is invariant under $Y \to -Y$ as well as, separately, under $T \to -T$. In (A) and $(B)$, $\mathcal R$ has two connected components—recall that the singular leaves, all at the conic tips of the red surfaces, do not belong to $\mathcal R$. In (C), $\mathcal R$
					has four  connected components} \label{fig:redzones}
			\end{figure}
			\noindent For a generic choice of $U$ and $V$, it becomes evident that symplectic leaves are not entirely contained within the forbidden red zones, as illustrated by examining the function $f$ in Example \ref{ex:quad}. The corresponding surface $\mathcal R$ is illustrated in Fig.~\!\ref{fig:redzones}B and has little to do with the symplectic leaves of the Poisson bivector 
			\begin{equation} \label{Piquad}
				\Pi_{qua} = (3z^2-1) \, \partial_x \wedge \partial_y - x \, \partial_x \wedge \partial_z + y \, \partial_y \wedge \partial_z \,.  
			\end{equation}
			They are depicted in Figs.~\ref{fig:cquad1} and \ref{fig:cquad-2} for two values of the Casimir function. As mentioned in the general discussion above, $\mathcal R$ does not contain the singular leaves. For the bivector in equation \eqref{Piquad}, these singular leaves are located at $ (x,y,z)=(0,0,\pm \tfrac{1}{\sqrt{3}})$, corresponding to $X= \pm \tfrac{1}{\sqrt{3}}$, $Y=0$, and $T=0$. These two points are clearly visible in Fig.~\!\ref{fig:redzones}B; they coincide with the conic tips of the red surface, to which they do not belong.
			
			\noindent The red zones for Example \ref{PLG}, the Poisson-Lie group, are illustrated in Fig.~\!\ref{fig:redzones}C. These zones are quite intricate and challenging to visualize graphically, making it initially unclear whether some of the symplectic leaves, such as those depicted in Fig.~\!\ref{figeta}, lie within the four connected components of $\mathcal R$. However, this is not the case, as becomes evident when examining the intersection with the symplectic leaves of $(M,\Pi_{grp})$: the intersection of the three leaves shown in Fig.~\ref{figeta} with $\mathcal R$ is depicted in Fig.~\ref{RSgrp}.
			\begin{figure}[H]
				\subfloat[a][]
				\centering
				\scalebox{0.1}[0.1]{\includegraphics{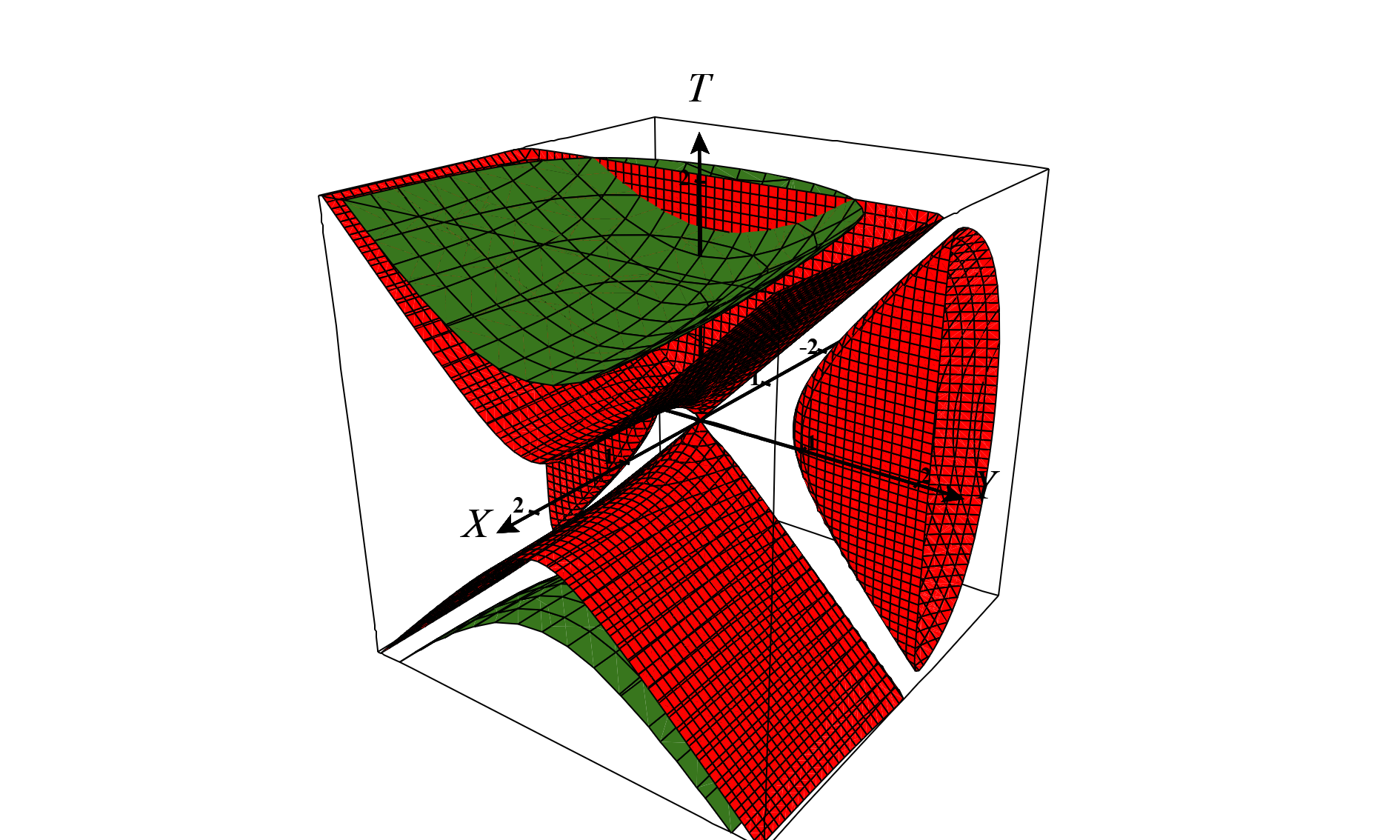}} 
				\hspace{-2.3cm}
				\subfloat[b][]
				{\includegraphics[scale=0.1]{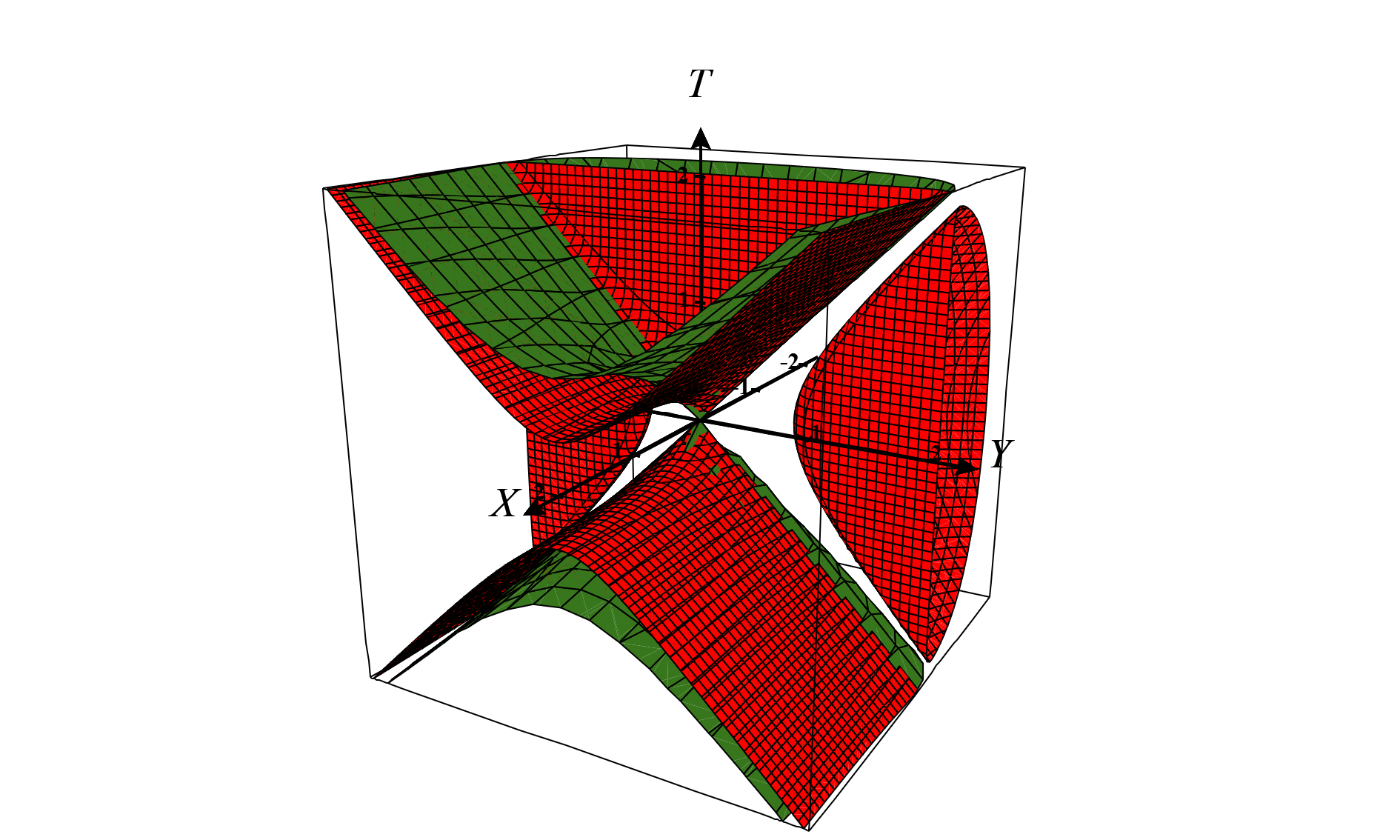}} 
				\hspace{-2.3cm}
				\subfloat[c][]
				{\includegraphics[scale=0.1]{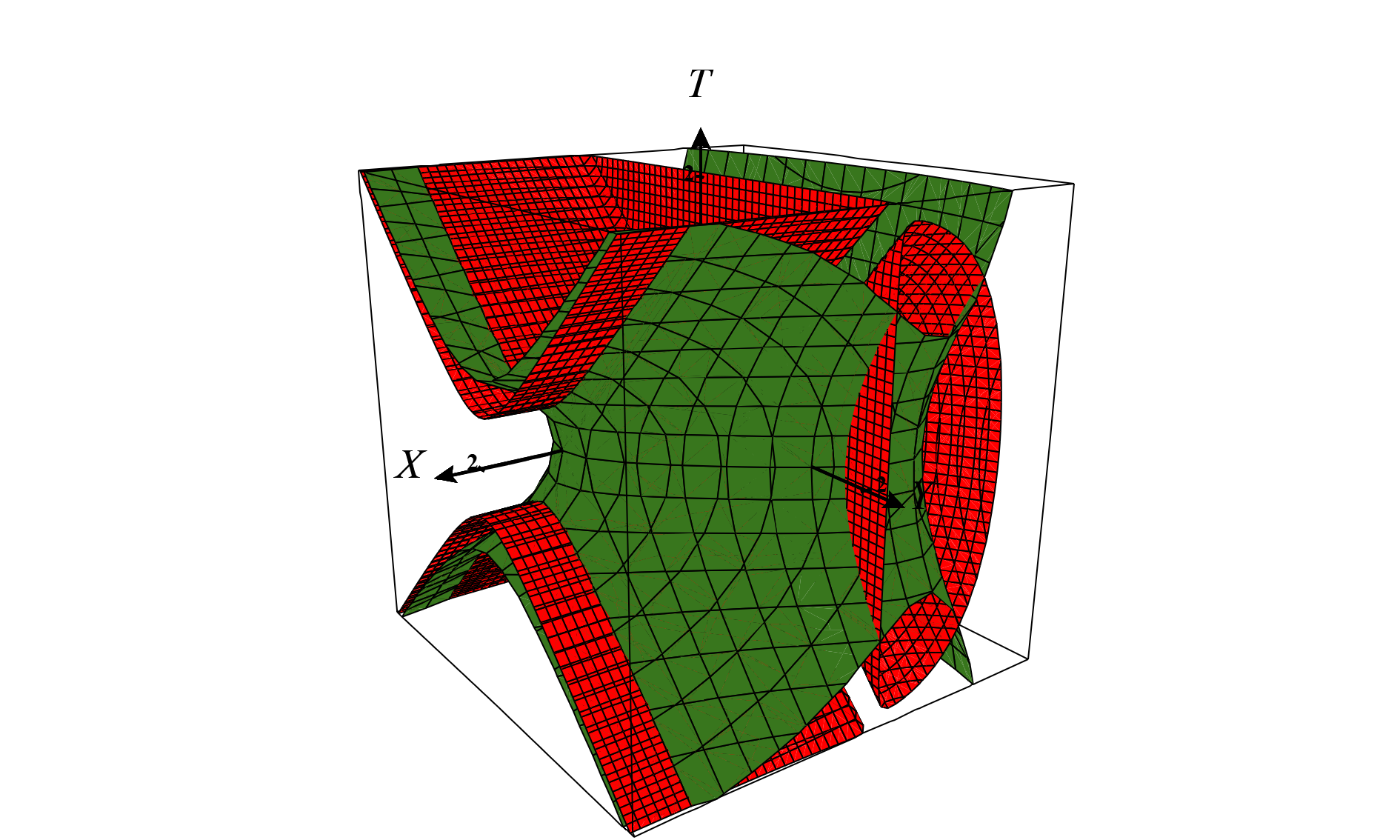}}
				\caption{Intersecting the red zone $\mathcal R$ of the Poisson Lie group, drawn in red for $\eta=1$, with the three symplectic leaves obtained for (A)  $c=-1$, (B) $c=0$, and (C) $c=1$, all drawn in green.  } \label{RSgrp}
			\end{figure}
		\end{exmp}
		
		\section{Geometric interpretation, green zones and red lines}\label{sec:7}
		In this section, we consider the class of Poisson structures given by (\ref{sl2nonlin}) together with the metric (\ref{g}). We discuss the geometric interpretation of symplectic leaves and their relationship to the induced metric on the Poisson manifold. The following points summarize the key ideas:
		
		\begin{itemize}
			\item  The geometry of the pictures allows us to determine whether a leaf \( S \) or a region inside \( S \) carries a non-degenerate induced metric.
			
			\item The coordinates \( X, Y, \) and \( T \) are used for visualization, while \( x, y, \) and \( z \) are retained for calculations due to their convenience.
			
			\item  Using the causal coordinates \( (X,Y,T) \), we identify the manifold \( M \) with a 2+1-dimensional Minkowski space \( \mathbb{R}^3 = \{(X,Y,T)\} \).
		\end{itemize}   
		\begin{exmp}
			The three symplectic leaves in Fig.~\ref{fig1} can be reinterpreted as follows: the leaf with \( T > 0 \) represents the ``future light cone," while the leaf with \( T < 0 \) corresponds to the ``past light cone." They are separated by ``present time," which is represented by the pointlike symplectic leaf at \( (0,0,0) \).
		\end{exmp}     
		
		Let us now recall some fundamental facts and terminology from Minkowski space \( M \): A vector \( v \in \T M \) is termed \emph{time-like} if its length is negative, i.e., \( g(v,v) < 0 \); \emph{space-like} if \( g(v,v) > 0 \); and \emph{null} if \( g(v,v) = 0 \). Vectors are classified as time-like if they are ``essentially vertical" in our drawings, meaning they form an angle of less than 45 degrees with respect to the \( T \)-axis. They are considered null if this angle is exactly 45 degrees, and space-like if they are ``essentially horizontal," indicating an angle greater than 45 degrees. A curve is classified as time-like, space-like, or null if all its tangent vectors possess the corresponding characteristics.
		\begin{rem}
			A submanifold or part of a submanifold \( S \) is of Riemannian or Euclidean nature—meaning the induced metric \( g_{\ind}^{S} \) is positive definite—if all curves lying within \( S \) are space-like.
		\end{rem}   
		\noindent       This is evident, for example, in Fig.~\!\ref{fig1}A, and it also applies to all the yellow regions in the leaves of Fig.~\!\ref{figeta}. 
		\begin{rem}
			A submanifold \( S \), or a portion of it, is pseudo-Riemannian or Lorentzian if, at each point \( s \) within it, there exist curves that are both space-like and time-like.     
		\end{rem}
		\noindent  This is applicable, for instance, to the symplectic leaf depicted in Fig.~\!\ref{fig:cquad1}, as well as to at least part of the light green regions in the leaves of Fig.~\!\ref{figeta}.   

\smallskip
\begin{rem}
The one-dimensional lightlike distribution $\Delta: S \to \operatorname{Rad} \T S$ on a lightlike leaf $S$ is integrable to a line. Specifically, it is defined by the intersection of the red zone (\ref{R}) with the two-dimensional symplectic leaf $S$, and is therefore referred to as the \textbf{red line}\footnote{The ``red lines" might be referred to as ``red domain walls" for higher-dimensional almost regular leaves.}. Geometrically, a red line is a line within a lightlike leaf $S$ where a signature change occurs between Euclidean and Lorentzian green zones.
\end{rem}

		\begin{exmp}    
			Consider the leaf \( S \) depicted in Fig.~\!\ref{fig:cquad-2}. This leaf contains Euclidean regions—particularly at the top and bottom of the hole visible in Fig.~\!\ref{fig:cquad-2}A—as well as Lorentzian regions, such as the left side of Fig.~\!\ref{fig:cquad-2}A. The transition between these regions occurs along specific lines, which are identified as red lines, where the signature change takes place.
		\end{exmp}
		\begin{exmp}
			Let us consider the future and past light cones, as illustrated in Fig.~\ref{fig1}B and Fig.~\ref{fig:redzones}A. At every point \( m \) on these cones, there exists a null curve (consider the straight line connecting the origin to \( m \)), as well as numerous space-like curves. However, there are no time-like curves passing through \( m \). The zero length of the null tangent vector can be explained by the fact that such a vector becomes an eigenvector with an eigenvalue of zero for the induced bilinear form \( g_{\ind}^S \) at \( m \). In this case, we are dealing with entirely bad leaves.
		\end{exmp}
		From the above qualitative discussion, the following observation is obvious:		
		\begin{obs} 
			However, leaves which are $\M$-singular \emph{everywhere} (``bad leaves") are very exceptional. More often they will be either good leaves—leaves which do not have an intersection with the red zone $\mathcal R$—or leaves of mixed nature, where Lorentzian and Euclidean regions are separated by red lines of $\M$-degenerate points.\footnote{More generally, one may want to call leaves $S$ where $\M$-degenerate points form a subset of measure zero ``almost good leaves": there some care will be needed when approaching the forbidden red walls, but within a good region, the conditions for Theorem \ref{gradient th} to hold true are satisfied.}
		\end{obs}
		Recall that the admissible, i.e., \( \mathcal{M} \)-regular parts of a symplectic leaf of interest defines the green zone, and each intersection of the leaf with the red zone is a red line. 
		Let us illustrate this with Example \ref{ex:quad}: 
		\begin{exmp}        Fig.~\ref{fig:cquad1} shows a leaf \( S \) that is endowed with a Lorentzian metric \( g^S \). The absence of any \( \mathcal{M} \)-singular points is further confirmed by Fig.~\!\ref{fig:redgreenquad}A, which displays \( S \) (the green surface) alongside \( \mathcal R \) (the red surface); they do not intersect.  This changes for the symplectic leaf \( S' \) of Fig.~\!\ref{fig:cquad-2}: Fig.~\!\ref{fig:redgreenquad}B shows that \( S' \) has non-empty intersections with \(\mathcal R \). By keeping only \( S' \) and the parts of \( \mathcal R \) that intersect with \( S' \), we obtain Fig.~\ref{fig:redlinesquad} 
			\begin{figure}[H]
				\subfloat[B][]
				{\includegraphics[scale=0.1]{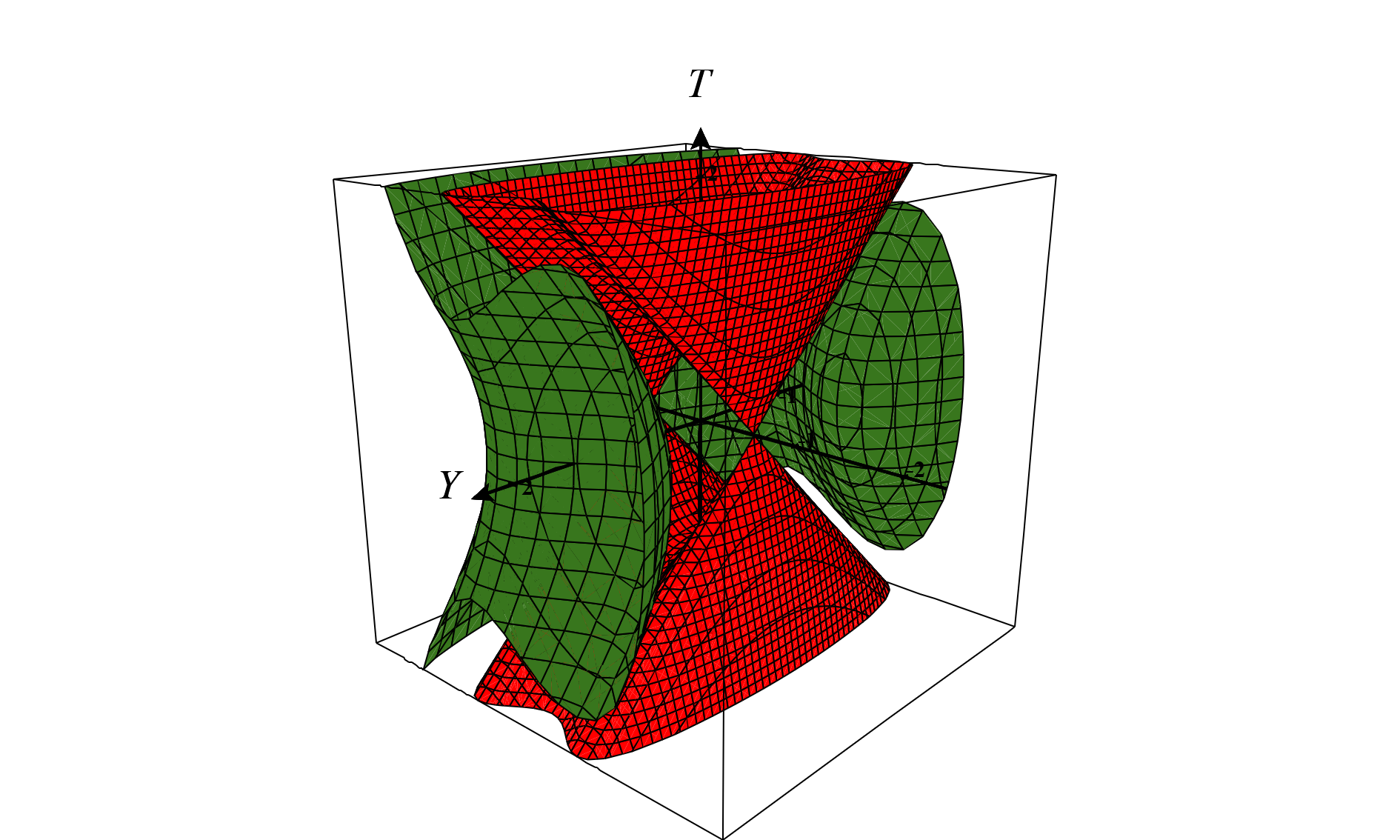}} 
				\hspace{-2.3cm}
				\subfloat[c][]
				{\includegraphics[scale=0.1]{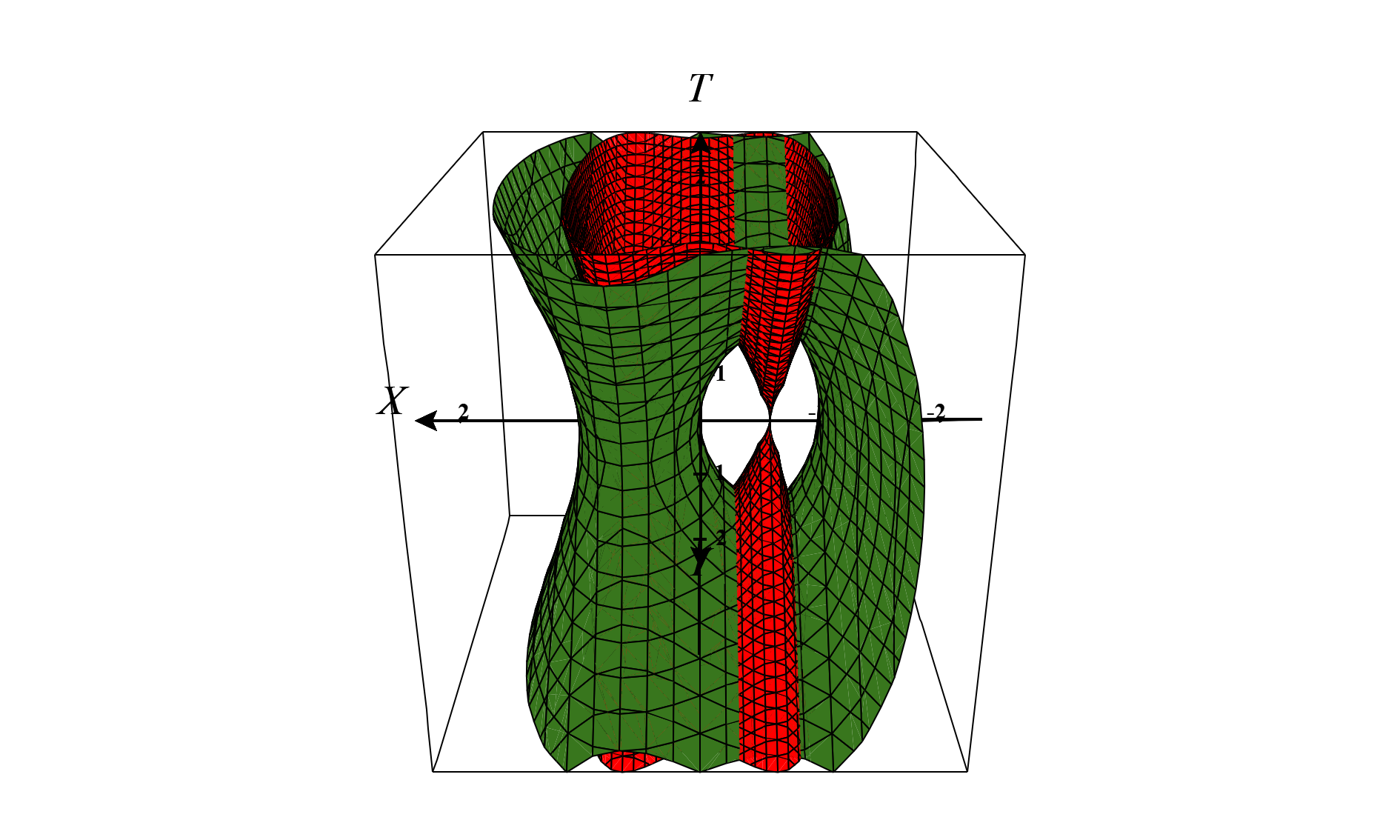}}
				\caption{Intersecting the red zone $\mathcal R$ of $(\R^3,g,\Pi_{qua})$, represented in red, with the two symplectic leaves, depicted in green, obtained for  (A)  $c=1$ and (B) $c=0$. While in (A) there are no intersections, in (B)  intersections do occur (see also Fig.~\ref{fig:redlinesquad} below).} \label{fig:redgreenquad}
			\end{figure}   
		\end{exmp} 
		\begin{figure}[H]
			\centering \includegraphics[width=0.5\linewidth]{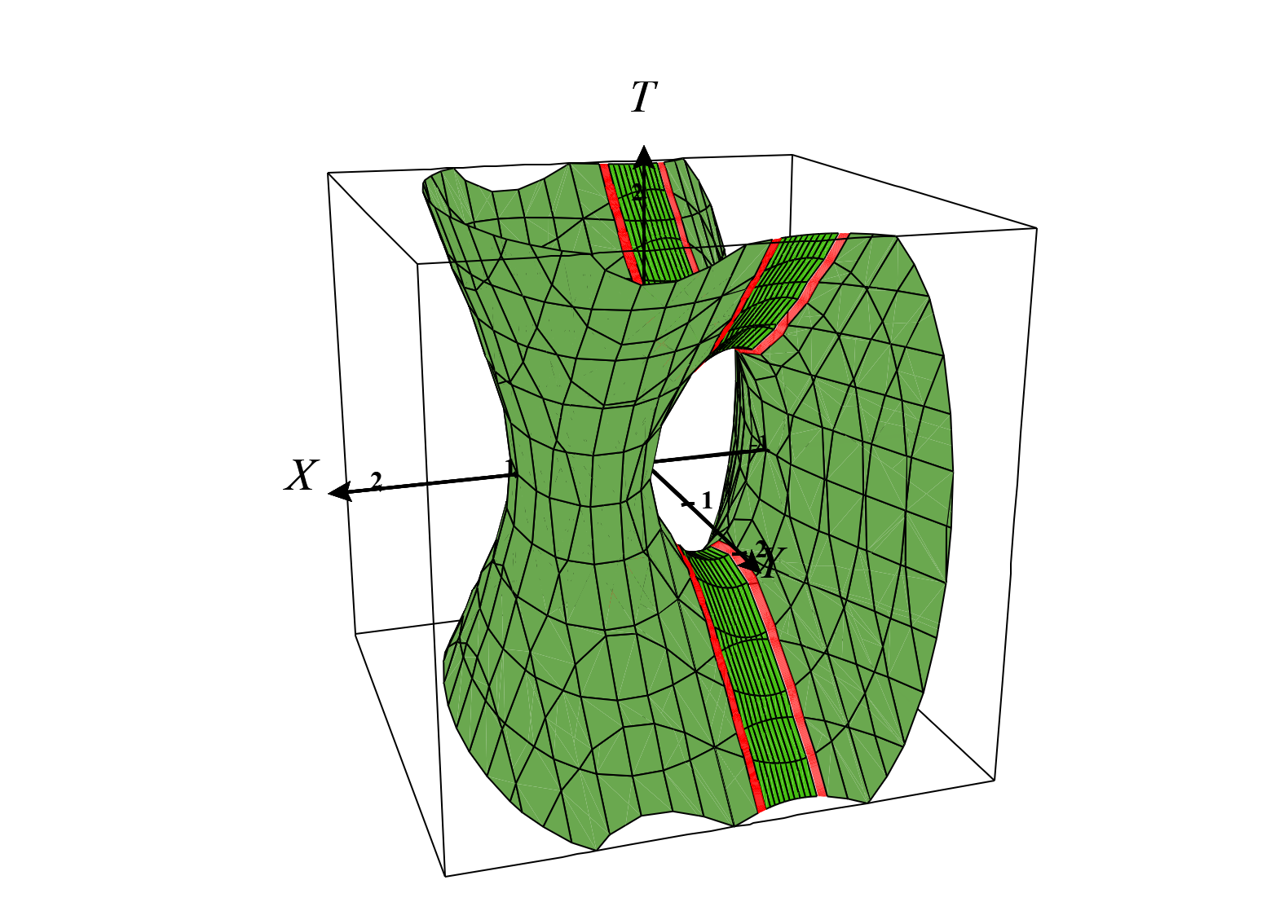}
			\caption{The green zones and red lines for the lightlike symplectic leaf corresponding to $c=0$ in the example with quadratic brackets. The dark green region between the two red lines represents an area of Euclidean signature, while the light green regions to the left and right of the lines indicate areas of Lorentzian signature.}
			\label{fig:redlinesquad}
		\end{figure}
		
		In the class of examples discussed in this paper, there is a nice way of characterizing the red lines. Let us fix a symplectic leaf $S_c$ corresponding to the value $c$ of the Casimir function \eqref{C}.
		\begin{cor}
			On a symplectic leaf \( S_c \), the zeros $z_{red}$ of the function 
			\begin{equation} \label{Fc}
				F_c =  U^2 + 2(1+UV)e^{-P} (c-Q)  
				+ V^2 \, e^{-2P} (c-Q)^2\,,
			\end{equation}
			precisely determine the red lines on the leaf $S_c$. 
		\end{cor}
		\begin{proof}
			By (\ref{C}), on the symplectic leaf \( S_c \),  we can express \( xy \) as a function of $z$: $xy=\exp[-P(z)] \cdot \left[c-Q(z)\right]$. Plugging this
			into 
			\eqref{xy}, we obtain a function $F_c = f\vert_{S_c} \colon  \R \to \R$, which takes the form (\ref{Fc}). It is precisely the zeros \( z_{\text{red}} \) of the function \( F_c\) that determine the red lines on the leaf \( S_c \). These red lines consist of those points on \( S_c \) where the \( X \)-coordinate takes one of the values corresponding to the zeros of the function \( F_c \). 
			\begin{equation}\label{zred}
				X \equiv z = z_{red}. 
			\end{equation} 
		\end{proof}
		It is remarkable that, for \emph{every} choice of $U$ and $V$, the intersection of the red zone with \emph{any} symplectic leaf has a constant value for $X$. Note that the intersection of the planes \eqref{zred} with a symplectic leaf \( S \) yields the red lines only if \( S = S_c \). 
		\begin{com}   
			Just as the zeros of the function \( h_c \) determine the topological nature of the corresponding symplectic leaf \( S_c \) of \( (M, \Pi) \), the zeros of the function \( F_c \) determine the red lines and green zones of this leaf in \( (M, \Pi, g) \). Moreover, if \( F_c(z) > 0 \), then for this value of \( z \) or \( X \), the induced metric has Lorentzian signature; conversely, if \( F_c(z) < 0 \), it is Euclidean.
		\end{com}        
		\begin{exmp}
			Let us return to the example of quadratic brackets for an illustration again: The function \eqref{Fc} becomes $F_c=9z^4-2z^3-6z^2+2z+1+2c$. Its graph is drawn in Fig.~\!\ref{fig:F_c.quad} for the two values of \( c \) corresponding to the leaves depicted in Fig.~\!\ref{fig:cquad1} (\( c=1 \)) and Fig.~\!\ref{fig:cquad-2} (\( c=0 \)). From this diagram, we observe that for \( c=1 \), there are no zeros of this function, and it is strictly positive. This implies that the corresponding leaf \( S_0 \), as shown in Fig.~\!\ref{fig:cquad1}, is a good leaf and that the induced metric is of Lorentzian signature everywhere. While we have established this previously through other means, it is evident that this conclusion can be easily drawn from a simple inspection of the graph of the one-argument function \( F_1 \).
			\begin{figure}[H]
				\centering
				\includegraphics[width=0.5\linewidth]{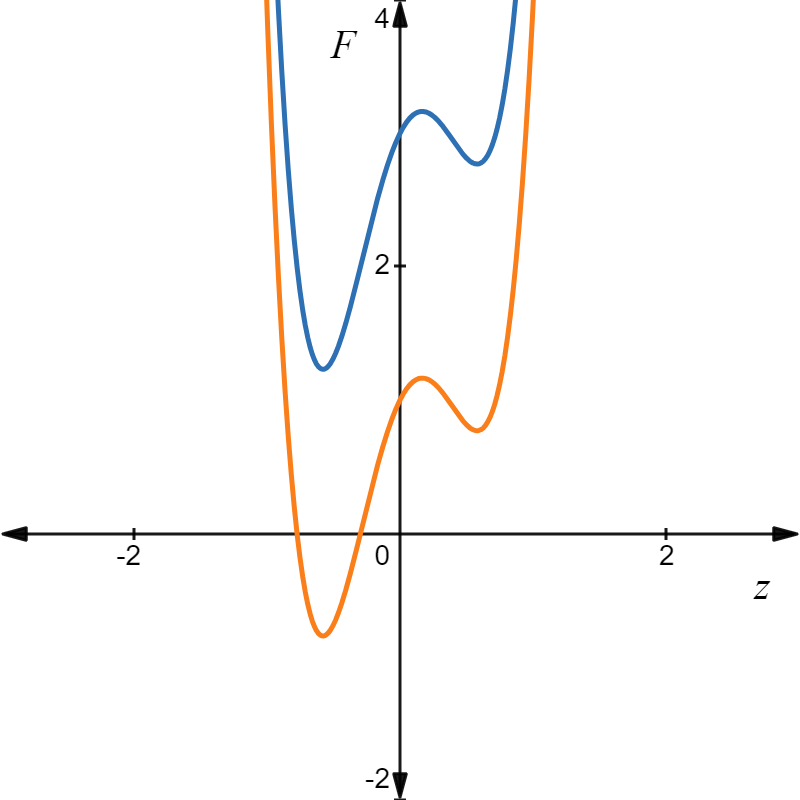}
				\caption{The function $F_c$ for the quadratic Poisson structure, orange color for $F_0$ and blue color for $F_1$.}
				\label{fig:F_c.quad}
			\end{figure}
			\noindent  	Likewise, we see that \( F_0 \) has two zeros located at the values \( z_{\text{red},1} \approx -0.77 \) and \( z_{\text{red},2} \approx -0.30 \). These zeros fix the location of the two red lines on the leaf \( S_0 \), as depicted in Fig.~\!\ref{fig:redlinesquad}. Between these two values of \( z \equiv X \), the function \( F_0 \) is negative, indicating that the region between the red lines, shown as dark green in Fig.~\!\ref{fig:redlinesquad}, is Riemannian. For values of \( X \) smaller than \( z_{\text{red},1} \) or greater than \( z_{\text{red},2} \), \( F_0 \) is positive; therefore, those green zones, depicted in light green in Fig.~\!\ref{fig:redlinesquad}, carry an induced Lorentzian signature metric.
		\end{exmp}
		
		\section{Gradient-Like behavior of GDB in green zones}\label{sec:8}
		We consider pseudo-Riemannian Poisson manifolds $(M,\Pi,g)$ defined by $M=\mathbb{R}^3$, the pseudo-Riemannian metric \eqref{g}, and the bivector \eqref{Pi}. In this section, we characterize the GDB vector fields for this class of Poisson structures, demonstrating that they are gradient vector fields in the green zones with respect to the corresponding DB metric. Specifically, we show that within these regions, the GDB vector fields exhibit properties consistent with those of gradient fields, enabling us to leverage their characteristics for further analysis. We proceed as follows:
	
	\medskip	
		\noindent {\bf To determine the DB metric on symplectic leaves}:  First, we need to determine the geometric data on a regular symplectic leaf 
		$S$, specifically the induced metric $g_{\ind}^S$ and the symplectic form $\omega^S$. We start by imposing the condition $\mathcal C(x,y,z)=c$ for (\ref{C}). Taking the differential of this equation, one has 
		\begin{equation}
			x \, \d y + y \, \d x + W \d z \approx 0\,.
		\end{equation}
		Here and in what follows, we use $\approx$ to denote equations valid exclusively on the symplectic leaf $S$. When using the coordinates $(x,z)$ on $S$, we have:
		\begin{equation} \label{diff}
			\d y \approx - \frac{ y \, \d x + W \d z }{x}   \,.
		\end{equation}
		This expression holds in regions of the leaf where $x \neq 0$. Similarly, when using $(y,z)$ coordinates on the leaf, we derive expressions valid for $y \neq 0$. The coordinates $(x,y)$ serve as suitable coordinates on $S$ when $W(x,y,z) \neq 0$. As the simultaneous vanishing of $x$, $y$, and $W$ corresponds precisely to singular point-like leaves, we can cover the entirety of $S$ using these three coordinate charts. For the remainder of our analysis, we will focus on the coordinate chart $(x,z)$, noting that analogous expressions can be obtained for the other two charts.
		
		\medskip \noindent
		Plugging \eqref{diff} into the embedding metric \eqref{g}, we  see that the induced metric becomes
		\begin{equation} \label{gindS}
			g_{\ind}^S \approx - \frac{2y}{x }\d x^2 - \frac{2W}{x} \d x \d z + \d z^2\,.
		\end{equation}
		in the chart with $x \neq 0$. Here $y$ is understood as the following function of $x$, $z$, and the parameter $c$:
		\begin{equation} \label{y}
			y \approx \frac{e^{-P(z)} \, \left( c -Q(z)\right)}{x} =: y(x,z)\, ,
		\end{equation}
		where $P$ and $Q$ are the functions defined in Lemma \ref{lemC}. The coordinate $y$ also enters $W$, so in \eqref{gindS} 
		\begin{equation}  \label{Wy}
			W \approx U(z) + 2x \, y(x,z) \, V(z)\,.   
		\end{equation}
		
		\medskip \noindent
		The apparent singularity of $g_{\ind}^S$ arises from the fact that the coordinate system $(x,z)$ on $S$ breaks down when $x=0$. We note in parenthesis that this corresponds to the plane $Y=T$ in the coordinate system $\eqref{XYT}$. Additionally, $g_{\ind}^S$ exhibits an inherent issue when a point on the leaf $S$ falls within the red zone $\mathcal R$, as referenced in $\eqref{R}$. Calculating the determinant of the induced metric, we obtain:
		
		\begin{equation}
			\det g_{\ind}^S \approx - \frac{f^S(x,z)}{x^2} \, , 
		\end{equation}
		where  $f^S(x,z):=f(x,y(x,z),z)$ with the function $f$ as defined in \eqref{f}. So, as expected, $\det g_{\ind}^S$ vanishes on the red lines. This corresponds to 
		\begin{equation} \label{ker}
			\ker (g_{\ind}^S)^{\sharp} \vert_m \approx \mathrm{Vect} (x\partial_x\vert_m + W\partial_z\vert_m)\,, \qquad\qquad \forall m \in \mathcal R \cap S\, . 
		\end{equation}

		\medskip \noindent
		On the other hand, from the second Poisson bracket in \eqref{sl2nonlin}, we deduce that 
		\begin{equation}
			\omega^S \approx \frac{\d x \wedge \d z}{x} \, . 
		\end{equation}
		Indeed, $\omega^S$ is well-defined on the entire leaf $S$. The apparent singularity at $x=0$ is merely a coordinate singularity. We now proceed to determine the DB metric \eqref{DB}. A direct calculation yields
		\begin{equation} \label{combine1}
			\tau_{DB} \approx \frac{2y\d x \d x+ 2W \, \d x \d z-x\d z\d z}{x(2xy+W^2)} \, ,
		\end{equation}
		with $y$ and $W$ as in \eqref{y} and \eqref{Wy}.
		We remark that \begin{equation}
			\tau_{DB} \approx -\frac{1}{f^S} \; g_{\ind}^S \, .
			\label{tauDBsing}
		\end{equation}
		Evidently, this tensor is not well-defined on the red lines, i.e.\ for points $m \in \mathcal R \cap S$—and it also does not have a continuous continuation into such points. In contrast, $g_{\ind}^S$ is a well-defined tensor on all of $S$, though it fails to define a pseudo-metric on $\mathcal R \cap S$ due to the presence of a kernel (see \eqref{ker}). However, in the green zones, defined as $S \backslash (\mathcal R \cap S)$, both $g_{\ind}^S$ and $\tau_{DB}$ successfully define a (pseudo) metric.

		\medskip \noindent
		{\bf To determine GDB vector fields:}	
		The GDB vector field $\partial_ G$ is defined on the entire manifold $M=\mathbb{R}^3$. To determine it for a function $G \in C^\infty(\mathbb{R}^3)$, we apply the negative of the matrix \eqref{Mmatrix} to the gradient vector $[\d G] = (G,_x,G,_y, G,_z)$, where the comma notation denotes partial derivatives. This calculation yields:
		\begin{align}
			\partial_ G 
			= \: & \left(- x^2G,_x +  \left[xy+W^2\right]G,_y - xWG,_z\right) \frac{\pa}{\pa x}\, + \nonumber\\ \: & \left( -y^2G,_y+\left[xy+W^2\right]G,_x - yWG,_z\right) \frac{\pa}{\pa y}
			\, + \label{pG2}\\ \: & 
			\left(-xWG,_x  - yWG,_y +2xyG,_z\right) \frac{\pa}{\pa z} \, . \nonumber 
		\end{align}
		To illustrate how the GDB vector field arises from more geometric quantities, we present the three elementary Hamiltonian vector fields corresponding to the canonical coordinates $(x,y,z) \in \mathbb{R}^3$:
		\begin{equation}\label{Hamxyz1}
			X_x =  W \frac{\pa}{\pa y} -x \frac{\pa}{\pa z}\, , \quad
			X_y = -W \frac{\pa}{\pa x} + y \frac{\pa}{\pa z} \, , \quad 
			X_z = x \frac{\pa}{\pa x} -y \frac{\pa}{\pa y}  \, . 
		\end{equation} 
		Now, ${g}^\flat$ applied to any vector field $a\frac{\pa}{\pa x} + b\frac{\pa}{\pa y} + c\frac{\pa}{\pa z}$ yields the 1-form $b \d x + 
		a\d y + c\d z$, and since $\Pi^{\sharp}( \d x )= X_x$ etc, we easily find the image of \eqref{Hamxyz1} under $\Pi^{\sharp} \circ {g}^\flat$ to be
		\begin{align}\label{Hamxyz2}
			\Pi^{\sharp} \left( {g}^\flat(X_x) \right)  &=   W X_x -x X_z\, ,\nonumber \\
			\Pi^{\sharp} \left( {g}^\flat (X_y) \right) &= -W X_y + y X_z\, , \\ 
			\Pi^{\sharp} \left( {g}^\flat (X_z) \right)  &= \: \;\: \, x X_y -y X_x  \, .\nonumber
		\end{align}   
		Plugging \eqref{Hamxyz1} into \eqref{Hamxyz2}, we obtain the three fundamental GDB vector fields $\partial_x$, $\partial_ y$, and $\partial_ z$, respectively. 
		
		\medskip \noindent  The vector field $\partial_G$ is defined in $\mathbb{R}^3$ but is tangent to the symplectic leaves, including the specific leaf $S$. We can therefore consider the restriction $\partial_G|S$ as a vector field on $S$. To understand its representation in the coordinate system $(x,z)$ on $S$, we first express $\partial_ G \in \Gamma(\mathbb{R}^3)$ in coordinates adapted to the symplectic leaves.
		
		\noindent  We introduce a new coordinate system:
		$\widetilde{x} := x$, $\widetilde{y} := \mathcal C(x,y,z)$, and $\widetilde{z} := z$. This coordinate system is well-defined on $\mathbb{R}^3 \setminus ({0} \times \mathbb{R}^2)$. Under this change of coordinates, the partial derivative $\frac{\pa}{\pa x}$ transforms as follows:
		$$ \frac{\pa}{\pa x} = \frac{\pa}{\pa \widetilde x} + C,_y \frac{\pa}{\pa \widetilde y} \, .$$
		This transformation occurs despite $x = \widetilde{x}$, due to the dependence of $\widetilde{y}$ on $x$ through the function $\mathcal C(x,y,z)$. However, since the vector field $\partial_G$ is tangent to the symplectic leaf $S$ when restricted to it, all contributions proportional to $\frac{\pa}{\pa \widetilde{y}}$ vanish, while those proportional to $\frac{\pa}{\pa x}$ and $\frac{\pa}{\pa z}$ remain unchanged. Consequently, in this new coordinate system, the vector field \eqref{pG2} retains its form after replacing the untilded coordinates with tilded ones and lacing the untilded coordinates with tilded ones and, at the same time, simply dropping the second line on the right-hand side.
		
		\medskip \noindent Therefore, after clarifying the calculations and once again using the coordinates $x=\widetilde x$  and $z=\widetilde z$ on the leaf $S$, we obtain
		\begin{equation} \label{combine2}
			\partial_G|_S  \approx  \left( -x^2G,_x +  \left[xy+W^2\right]G,_y - xWG,_z\right) \frac{\pa}{\pa x}+ 
			\left(-xWG,_x  -yWG,_y +2xyG,_z\right) \frac{\pa}{\pa z} \,.
		\end{equation}
		Here, the derivatives of $G$ are computed before restricting to $S$, and, as previously mentioned, the functions $y$ and $W$ are defined by equations \eqref{y} and \eqref{Wy}, respectively.
		
		\medskip \noindent 
		{\bf GDB vector fields as gradient vector fields:}		We are now ready to combine the two main components, \eqref{combine1} and \eqref{combine2}. After a somewhat tedious calculation and the cancellation of several terms, we obtain:
		\begin{equation}
			\tau_{DB} \left({\partial_ G}|_S ,\cdot\right) \approx - G,_x \d x + G,_y \left(\frac{y}{x} \d x + \frac{W}{x} \d z\right)  -G,_z \d z \, .
		\end{equation}
		Upon usage of \eqref{diff}, the term in brackets following $G,_y$ is recognized to be precisely $-\d y$. This implies 
		\begin{equation} \label{equation}
			\tau_{DB} \left({\partial_ G}|_S ,\cdot\right) \approx -  \d \left( G\vert_S\right) \, , 
		\end{equation}
		which is equivalent to Equation \eqref{thmeq1}. 
		
		\medskip\noindent{\bf Geometric interpretation:}	In the above manipulations, it is understood that we are operating within the green zone of \( S \), as otherwise, \( \tau_{DB} \) would not be defined (cf. \eqref{tauDBsing} and the discussion that follows). Moreover, we observe that, according to the right-hand side of \eqref{equation}, the left-hand side has a continuous continuation to \( \mathcal{M} \)-singular points. This can be explained as follows: Recall that \( \partial_ G \) is a well-defined vector field in all \( \mathbb{R}^3 \), and thus the right-hand side of equation \eqref{combine2} is defined in \( \mathcal R \cap S \) provided that \( x \neq 0 \) (which ensures the applicability of our coordinate patch). Using this on \( \mathcal R \), we can replace \( -2xy \) with \( W^2 \); see equations \eqref{f} and \eqref{R}. Consequently, we find that in our chart \( (x,z) \) on \( S \):
		\begin{equation}
			\partial_ G|_{\mathcal R\cap S} \approx (x G,_x + y G,_y + W G,_z) (x \frac{\pa}{\pa x} + W \frac{\pa}{\pa z}).
		\end{equation}
		
		\noindent In $\mathcal{R}$, the image of $\mathcal{M}^\sharp$ becomes one-dimensional, and when restricted to the leaf $S$, it coincides with the kernel of $\tau_{DB}^S$; see equation \eqref{ker}. Consequently, $\tau_{DB}$ blows up precisely in the red zone $\mathcal{R}$; see Equation \eqref{tauDBsing}. Moreover, the vector field $\partial_G|_{\mathcal{R} \cap S}$ lies within the null space $\operatorname{Rad} \T S$, indicating that at each point, the flow of the GDB vector field is tangent to the null direction. In other words, on the red lines, the GDB vector field is a globally null vector field.

		\bigskip
		\noindent {\bf Acknowledgements.} 
		Z. Ravanpak acknowledges the ``Cercetare postdoctoral\u a avansat\u a” funded by the West University of Timi\c soara, Romania, the financial support from the Spanish Ministry of Science and Innovation under grants PID2022-137909NB-C22, and Erwin Schrödinger International Institute for Mathematics and Physics (ESI), University of Vienna,  where a part of this work has been done during her research stay as a part of program ``Geometry of Gauge Theories-SRF0324-2024". A part of this work was also conducted during the authors’ research stay at ESI as part of the program ``Infinite-dimensional Geometry: Theory and Applications-LTV-2025.”  The authors are grateful to Thomas Strobl for valuable remarks and suggestions to improve the paper.

	\end{document}